\documentclass[10pt]{article}

 \usepackage{amsmath}
\usepackage{amsthm,amssymb}
 \usepackage{makeidx,epsfig,lscape}
 \usepackage{color,colortbl}
 \usepackage{fancyhdr}
 \usepackage{times}
 \usepackage{xcolor,pict2e}
 \usepackage{indentfirst}
 \usepackage[latin1]{inputenc}
 
\newcommand{\R}{\mathbb R}
\newcommand{\N}{\mathbb N}

\newcommand{\Pv}{\mathbb P}
\newcommand{\E}{\mathbb E}

 \renewcommand{\headrulewidth}{0pt}
 \renewcommand{\footrulewidth}{0.5pt}

 \definecolor{myaqua}{rgb}{0.0,0.5,0.55}
 \definecolor{lightaqua}{rgb}{0.75,0.95,0.95}

 \usepackage[colorlinks = true,
            linkcolor = myaqua,
            urlcolor  = blue,
            citecolor = red,
            pdfpagemode=UseOutlines,
            bookmarksnumbered=true,
            pdfpagelabels,
            breaklinks]{hyperref}

\usepackage{caption}
\usepackage{floatrow}
\newtheorem{theorem}{Theorem}
\newtheorem{prop}{Proposition}
\newtheorem{lem}{Lemma}
\newtheorem{coro}{Corollary}
\newtheorem{defn}{Definition}[section]

\newtheorem{exples}{Examples}[section]
\newtheorem{rem}{Remark}[section]
\def\lin#1#2{\textcolor[rgb]{0.6,0.6,0.6}{\vspace*{#1mm} \hrule
   height 3 pt \vspace*{#2mm}}}
\def\bt{\begin{tabular}}
\def\et{\end{tabular}}
\def\and{\mbox{ and }}
\def\E{\mbox{\bf E}}
\def\P{\mbox{\bf P}}

\def\1{{\bf 1}}

 \def\boxx#1#2#3#4#5{
 {\linethickness{#4pt}\put(#1,#5){\color{myaqua}{\line(1,0){#3}}}}
 \multiput(#1,#2)(0,#4){2}{\line(1,0){#3}}
 \multiput(#1,#2)(#3,0){2}{\line(0,1){#4}}
  }

\begin{document}

 $\mbox{ }$

 \vskip 12mm

{ % \fontfamily{Cambria}\selectfont.

% "Title of the Paper"
{\noindent{\Large\bf\color{myaqua}
   A new class of stochastic processes with great potential for interesting applications.}
%
% \runtitle{Change-Point Analysis of Survival Data}
\\[6mm]
{\bf  Fulgence EYI OBIANG$^{1,a}$, Paule Joyce MBENANGOYE$^{1,b}$, Magloire Yorick NGUEMA MBA$^{1,c}$ and Octave MOUTSINGA$^{1,d}$}}
\\[2mm]
{ % \\ fontfamily{Calibri} \\ selectfont
$^1$URMI Laboratory, Département de Mathématiques et Informatique, Faculté des Sciences, Université des Sciences et Techniques de Masuku, BP: 943 Franceville, Gabon. 
\\
$^a$ Email: \href{mailto:feyiobiang@yahoo.fr}{\color{blue}{\underline{\smash{feyiobiang@yahoo.fr}}}}\\[1mm]
$^b$ Email: \href{mailto:paulejoycembenangoye@yahoo.fr}{\color{blue}{\underline{\smash{paulejoycembenangoye@yahoo.fr}}}}\\[1mm]
$^c$ Email: \href{yorickmagloire@gmail.com}{\color{blue}{\underline{\smash{yorickmagloire@gmail.com}}}}\\[1mm] 
$^d$ Email: \href{mailto:octavemoutsing-pro@yahoo.fr}{\color{blue}{\underline{\smash{octavemoutsing-pro@yahoo.fr}}}}\\[1mm] 
%$^3$ Email:\href{mailto:ouknine@ucam.ac.ma}{\color{blue}{\underline{\smash{ouknine@uca.ac.ma}}}}\\[1mm]
\lin{5}{7}

 {  %\fontfamily{Cambria}\selectfont.
 {\noindent{\large\bf\color{myaqua} Abstract}{\bf \\[3mm]
 \textup{
 This paper contributes to the study of a new and remarkable family of stochastic processes that we will term class $\Sigma^{r}(H)$. This class is potentially interesting because it unifies the study of two known classes: the class $(\Sigma)$ and the class $\mathcal{M}(H)$. In other words, we consider the stochastic processes $X$ which decompose as $X=m+v+A$, where $m$ is a local martingale, $v$ and $A$ are finite variation processes such that $dA$ is carried by $\{t\geq0:X_{t}=0\}$ and the support of $dv$ is $H$, the set of zeros of some continuous martingale $D$. First, we introduce a general framework. Thus, we provide some examples of elements of the new class and present some properties. Second, we provide a series of characterization results. Afterwards, we derive some representation results which permit to recover  a process of the class $\Sigma^{r}(H)$ from its final value and of the honest times  $g=\sup\{t\geq0:X_{t}=0\}$ and $\gamma=\sup{H}$. In final, we investigate an interesting application with processes presently studied. More precisely, we construct solutions for skew Brownian motion equations using stochastic processes of the class $\Sigma^{r}(H)$.
 }}
 \\[4mm]
 {\noindent{\large\bf\color{myaqua} Keywords:}{\bf \\[3mm]
 class $(\Sigma)$; class $(\Sigma^{r})$; Signed measure theory; Honest time; Relative martingales.
}}\\[4mm]{\noindent{\large\bf\color{myaqua} MSC:}{\color{blue} 60G07; 60G20; 60G46; 60G48}}
\lin{3}{1}

\renewcommand{\headrulewidth}{0.5pt}
\renewcommand{\footrulewidth}{0pt}

 \pagestyle{fancy}
 \fancyfoot{}
 \fancyhead{} % clear all header and footer fields
 \fancyhf{}
 \fancyhead[RO]{\leavevmode \put(-160,0){\color{myaqua} Fulgence EYI OBIANG et al. (2022)} \boxx{15}{-10}{10}{50}{15} }
 %\fancyhead[LE]{\leavevmode \put(0,0){\color{myaqua}F. EYI-OBIANG et al (2015)}  \boxx{-45}{-10}{10}{50}{15} }
 \fancyfoot[C]{\leavevmode
 %\put(0,0){\color{lightaqua}\circle*{34}}
 %\put(0,0){\color{myaqua}\circle{34}}
 \put(-2.5,-3){\color{myaqua}\thepage}}

 \renewcommand{\headrule}{\hbox to\headwidth{\color{myaqua}\leaders\hrule height \headrulewidth\hfill}}
\section*{Introduction}

In this paper, we investigate the family of stochastic processes of the form: 
\begin{equation}\label{e}
	X=m+v+A,
\end{equation}
where $m$ is a local martingale, $v$ and $A$ are finite variation processes such that $dv$ is carried by $H$, the set of zeros of a given continuous martingale $D$ and the support of the signed measure $dA$ is the set $\{t\geq0:X_{t}=0\}$. We shall term this family, class $\Sigma^{r}(H)$. This class is potentially interesting because such processes play a central role in several probabilistic studies. Moreover, two important subclasses of this family of stochastic processes are already known and studied in the literature. They are classes $(\Sigma)$ and $\mathcal{M}(H)$. Specifically, the first one is the set of stochastic processes $X$ of the form: $X=m+A$, where $m$ and $A$ satisfy conditions given in Identity \eqref{e}. This notion of class $(\Sigma)$ was first defined by Yor and has been studied extensively by several authors, including Yor, Najnudel, Nikeghbali, Cheridito, Platen, Ouknine, Bouhadou, Sakrani, Eyi Obiang, Moutsinga, and Trutnau  (see \cite{siam,pat,eomt,fjo,naj,naj1,naj2,naj3,nik,mult,y1}). Some well-known examples of such processes include càdlàg local martingales, the absolute value of a continuous martingale, the positive and negative parts of a continuous martingale, solutions of skew Brownian motion equations starting from zero, and the drawdown of a càdlàg local martingale with only negative jumps. In addition, this class is very suitable for applications. For example: Nikeghbali in \cite{nik}, uses them to construct solutions for the Skorokhod's imbedding problem. Additionally, Eyi Obiang et al. in \cite{fjo}, construct weak solutions of certain differential stochastic equations from such processes. 

By contrast, the class $\mathcal{M}(H)$ is the family of stochastic processes $X$ taking the form: $X=m+v$, where $m$ and $v$ are defined as in Identity \eqref{e}. Well-known examples include relative martingales defined by Azema and Yor (class $\mathcal{R}(H)$) \cite{1} and local martingales. In addition, if $D$ is a Brownian motion, an another example is the geometric It\^o-Mckean skew Brownian motion process with Azzalini skew normal distribution 
$X^{\delta}=\sqrt{1-\delta^{2}}B+\delta|D|$, where $B$ is a Brownian motion independent of $D$. This last-mentioned example plays a capital role in many studies.  For instance, Corns and Satchell \cite{corn} and Zhu and He \cite{zhu} worked on this type of skew Brownian motion and priced European style options.
In fact, this class was recently defined and studied in \cite{mh}, where the authors provide a general framework, investigate stochastic differential equations driven by such processes and construct solutions for homogeneous and inhomogeneous skew Brownian motion equations by using processes of the last-mentioned class.

However, remark that the class $\Sigma^{r}(H)$ contains also elements which are not in $(\Sigma)\cup\mathcal{M}(H)$. An example of such processes is $|X^{\delta}|$, the absolute value of the above geometric It\^o-Mckean skew Brownian motion process $X^{\delta}$. There no already exist references studying processes of this  part of the class $\Sigma^{r}(H)$.

The aim of this paper consists in two points. The first one, is to propose a unified study for all stochastic processes of the class $\Sigma^{r}(H)$. Thus, we provide a general framework and methods for dealing with such processes. The second point, is to show that processes of the presently considered class can be useful to develop some applications. Hence, the remaining parts of this work are structured in the following manner: In Section \ref{sec:1}, we define notations and recall some useful preliminaries. In Section \ref{sec:2}, we provide some interesting examples and prove some structural properties satisfied by processes of the class $\Sigma^{r}(H)$. In Section \ref{sec:3}, we derive a series of characterization results for the class $\Sigma^{r}(H)$. In Section \ref{sec:4}, we obtain some formulas which permit to recover a process $X$ of the class $\Sigma^{r}(H)$ from its final value $X_{\infty}$ and of  honest times $\gamma=\sup\{s\leq t:s\in H\}$ and $g=\sup\{t\geq0:X_{t}=0\}$. These formulas are inspired of the one Azema and Yor have obtained for relative martingales (Theorem of \cite{1}) and the one Cheridito et al. have derived for processes of the class $(\Sigma)$ (Theorem of \cite{pat}). More precisely, we shall prove that under some assumptions, processes of the class $\Sigma^{r}(H)$ can be written as follows:
$$X_{t}=\E\left[X_{\infty}1_{\{g<t\}}|\mathcal{F}_{t}\right]+\E\left[(v_{t}-v_{d_{t}})1_{\{\gamma>t\}}|\mathcal{F}_{t}\right]$$
and
$$X_{t}=\E\left[X_{\infty}1_{\{\gamma<t\}}|\mathcal{F}_{t}\right]+\E\left[(A_{t}-A_{d^{'}_{t}})1_{\{g>t\}}|\mathcal{F}_{t}\right],$$
where, $d_{t}=\inf\{s>t:X_{s}=0\}$ and $d^{'}_{t}=\inf\{s>t:s\in H\}$. In fact, we shall obtain corollaries which show that Proposition 2.2 of \cite{1} and Theorem 3.1 of \cite{pat} are particular cases of the above representation results. Finally in Section \ref{sec:5}, we show that processes of the class $\Sigma^{r}(H)$ can also have good applications. For this, we propose to construct solutions for homogeneous and inhomogeneous skew Brownian motion equations. More precisely, we construct solutions from continuous processes of the class $\Sigma^{r}(H)$ for the following equations:
\begin{equation}
	X_{t}=x+B_{t}+(2\alpha-1)L_{t}^{0}(X)
\end{equation}
and
\begin{equation}
	X_{t}=x+B_{t}+\int_{0}^{t}{(2\alpha(s)-1)dL_{s}^{0}(X)},
\end{equation}
where $B$ is a standard Brownian motion and $x=0$. It must be remarked that solutions had already been built from the processes of the class $(\Sigma)$ (see \cite{fjo}). This should not be seen as a redundancy because some processes of the class $\Sigma^{r}(H)$ are not elements of the class $(\Sigma)$. For instance,  $X^{\delta}$ and $|X^{\delta}|$ are such examples.

%%%%%%%%%%%%%%%%%%%%%%%%%%%%%%%%%%%%%%%%%%%%%%%%%%%%%%%%%%%%%%%%%%%%%%%%%%%%%%%%%%%%%%%%%%%%%%%%%%%%%%%%%%%%%%%%%%%%%%%
\section{Recalling of useful preliminaries}\label{sec:1}

%%%%%%%%%%%%%%%%%%%%%%%%%%%%%%%%%%%%%%%%%%%%%%%
\subsection{Notations and some useful definitions}
In this work, we fix a filtered probability space $\left(\Omega,(\mathcal{F}_{t})_{t\geq0},\mathcal{F},\P\right)$ satisfying the usual conditions. Throughout, $H$ denotes the zero's set of a continuous martingale that we shall always term $D$. Thus, we shall use the following notations:
$$\gamma=\sup\{t\geq0:t\in H\}=\sup\{t\geq0:D_{t}=0\}\text{ and }\gamma_{t}=\sup\{s\leq t:s\in H\}=\sup\{s\leq t:D_{s}=0\}.$$
 And for any other process $X$, we shall denote $g=\sup\{t\geq0:X_{t}=0\}$ and $g_{t}=\sup\{s\leq t:X_{s}=0\}$. Remark that $\gamma$ and $g$ are not stopping time with respect to the filtration $(\mathcal{F}_{t})_{t\geq0}$ since they depend on the future. Such random variables are called honest times. And according to the enlargement filtration theory, there exist larger filtrations under which these random variables become stopping times. Thus, we will denote $(\mathcal{G}^{\gamma}_{t})_{t\geq0}$ and $(\mathcal{G}^{g}_{t})_{t\geq0}$ to represent the smaller filtrations under which $\gamma$ and $g$ are  respectively stopping times and such that $\forall t\geq0$, $\mathcal{F}_{t}\subset\mathcal{G}^{\gamma}_{t}$ and $\mathcal{F}_{t}\subset\mathcal{G}^{g}_{t}$.

On another hand, remark that for any continuous semi-martingale $Y$, the set $\mathcal{W}=\{t\geq0; Y_{t}=0\}$ cannot be ordered. However, the set $\R_{+}\setminus\mathcal{W}$ can be decomposed as a countable union $\cup_{n\in\N}{J_{n}}$ of intervals $J_{n}$. Each interval $J_{n}$ corresponds to some excursion of $Y$. In other words, if $J_{n}=]g_{n},d_{n}[$, $Y_{t}\neq0$ for all $t\in]g_{n},d_{n}[$ and $Y_{g_{n}}=Y_{d_{n}}=0$. For any constant $\alpha\in[0,1]$, we consider a sequence $(\zeta_{n})$ of i.i.d. Bernoulli variables such that
$$\Pv(\zeta_{n}=1)=\alpha\text{ and }P(\zeta_{n}=-1)=1-\alpha.$$
Now, let us define some progressive processes which will play a capital role in some parts of the present work (Section \ref{sec:3} and Section \ref{sec:5}).
\begin{equation}\label{zalpha}
	Z^{Y}_{t}=\sum_{n=0}^{+\infty}{\zeta_{n}1_{]g_{n},d_{n}[}(t)}.
\end{equation}
and
\begin{equation}\label{kalpha}
	k^{Y}_{t}=\sum_{n=0}^{+\infty}{\zeta_{n}1_{[g_{n},d_{n}[}(t)}.
\end{equation}

If we assume that $\alpha$ is a piecewise constant function associated with a partition $(0=t_{0}<t_{1}<\cdots<t_{n-1}<t_{m})$, i.e., $\alpha$ is of the form
$$\alpha(t)=\sum_{i=0}^{m}{\alpha_{i}1_{[t_{i},t_{i+1})}(t)},$$
where $\alpha_{i}\in[0,1]$ for all $i=0,1,\cdots,m$, then we shall consider the processes
\begin{equation}\label{Zalpha}
	\mathcal{Z}^{Y}_{t}=\sum_{n=0}^{+\infty}{\sum_{i=0}^{m}{\zeta^{i}_{n}1_{]g_{n},d_{n}[\cap[t-{i},t_{i+1})}(t)}},
\end{equation}
and 
\begin{equation}\label{Kalpha}
	\mathcal{K}^{Y}_{t}=\sum_{n=0}^{+\infty}{\sum_{i=0}^{m}{\zeta^{i}_{n}1_{[g_{n},d_{n}[\cap[t-{i},t_{i+1})}(t)}},
\end{equation}
where $(\zeta^{i}_{n})_{n\geq0}$, $i=1,2,\cdots,m$, are $m$ independent sequences of independent variables such that
$$\Pv(\zeta^{i}_{n}=1)=\alpha_{i}\text{ and }\Pv(\zeta^{i}_{n}=-1)=1-\alpha_{i}.$$

%%%%%%%%%%%%%%%%%%%%%%%%%%%%%%%%%%%%%%%%%%%%%%%%%%%%%%%%%%%%%%%
\subsection{Useful results}
Balayage formulas are important tools in this work. We recall some formulas in next:

\begin{prop}\label{balpro}
Let $Y$ be a continuous semi-martingale and $g_{t}=\sup\{s\leq t:Y_{s}=0\}$. Let $k$ be a bounded progressive process, where ${^{p}k_{\cdot}}$ denotes its predictable projection. Then,
\begin{equation}\label{bal}
	k_{g_{t}}Y_{t}=k_{0}Y_{0}+{\int_{0}^{t}{^{p}k_{g_{s}}dY_{s}}+R_{t}},
\end{equation}
where $R$ is an adapted, continuous process with bounded variations such that $dR_{t}$ is carried by the set $\{Y_{s}=0\}$.
\end{prop}

\begin{rem}
If the progressive process $k$ is càdlàg, we get ${^{p}k_{g_{s}}}={^{p}k_{s}}=k_{s-}$. Hence, according to the continuity of $Y$, \eqref{bal} becomes:
$$k_{g_{t}}Y_{t}=k_{0}Y_{0}+{\int_{0}^{t}{k_{s}dY_{s}}+R_{t}}.$$
\end{rem}

Proposition \ref{balpro} is a powerful and interesting tool. However, the fact that we know nothing about the form of the process $R$ can be limiting. Bouhadou and Ouknine \cite{siam} identified the process $R$ of Proposition \ref{balpro} when the progressive process $k$ is equal to progressive processes $Z^{Y}$ and $\mathcal{Z}^{Y}$ respectively defined in \eqref{zalpha} and \eqref{Zalpha}. We recall these results below.  

\begin{prop}[\textbf{Ouknine and Bouhadou \cite{siam}}]\label{pzalph}
Let $Y$ be a continuous semi-martingale and $Z^{Y}$ be the process defined in \eqref{zalpha}. Then,
$$Z^{Y}_{t}Y_{t}=\int_{0}^{t}{Z^{Y}_{s}dY_{s}}+(2\alpha-1)L_{t}^{0}(Z^{Y}Y),$$
where $L_{\cdot}^{0}(Z^{Y}Y)$ is the local time of the semi-martingale $Z^{Y}Y$.
\end{prop}

\begin{prop}[\textbf{Ouknine and Bouhadou \cite{siam}}]
Let $Y$ be a continuous semi-martingale and $Z^{Y}$ be the process defined in \eqref{Zalpha}. Then,
$$\mathcal{Z}^{Y}_{t}Y_{t}=\int_{0}^{t}{\mathcal{Z}^{Y}_{s}dY_{s}}+\int_{0}^{t}{(2\alpha(s)-1)dL_{s}^{0}(\mathcal{Z}^{Y}Y)},$$
where $L_{\cdot}^{0}(\mathcal{Z}^{Y}Y)$ is the local time of the semi-martingale $\mathcal{Z}^{Y}Y$.
\end{prop}

%%%%%%%%%%%%%%%%%%%%%%%%%%%%%%%%%%%%%%%%%%%%%%%%%%%%%%%%%%%%%%%%%%%%%%%%%%%%%%%%%%%%%%%%%%%%%%%%%%%%%%%%%%%%%%%%%%%%%%%
\section{Preliminary study of the new class}\label{sec:2}

Now, we start the study of processes of the class $\Sigma^{r}(H)$. The goal of this section is to provide a general framework for stochastic processes of this class. Thus, we define correctly the class $\Sigma^{r}(H)$. Afterwards, we give some examples which are not necessarily in the known classes $(\Sigma)$ and $\mathcal{M}(H)$. In final, we will end this section by exploring some general properties satisfied by elements of the class $\Sigma^{r}(H)$.

%%%%%%%%%%%%%%%%%%%%%%%%%%%%%%%%%%%%%%%%%%%%%%%%%%%%%%%%%%%%%%%%%
\subsection{Definition and examples}

We start this subsection by giving properly the definition of the class $\Sigma^{r}(H)$. Thus, we shall consider the following definition:
\begin{defn}\label{d1}
We say that a process $X$ is of the class $\Sigma^{r}(H)$ if it decomposes as $X=m+v+A$, where
\begin{enumerate}
	\item $m$ is a càdlàg local martingale, with $m_{0}=0$ ;
	\item $v$ and $A$ are adapted continuous finite variation processes such that $v_{0}=0$ and $A_{0}=0$; 
	\item $\int{1_{H^{c}}dv_{s}}=0$ and $\int{1_{\{X_{s}\neq0\}}dA_{s}}=0$.
\end{enumerate}
\end{defn}
 
Now, we shall show that processes of the class $\Sigma^{r}(H)$ decompose otherwise. To do this, we first define an other class that we term class $\Sigma_{r}(H)$.

\begin{defn}\label{d2}
We say that a process $X$ is of the class $\Sigma_{r}(H)$ if it decomposes as $X=M+A$, where
\begin{enumerate}
	\item $M\in\mathcal{M}(H)$, with $M_{0}=0$ ;
	\item $A$ is an adapted continuous process with finite variations such that $A_{0}=0$; 
	\item $\int{1_{H^{c}\cap\{X_{s}\neq0\}}dA_{s}}=0$.
\end{enumerate}
\end{defn}

The following proposition allows us to see that $\Sigma_{r}(H)$ and $\Sigma^{r}(H)$ are identical.
\begin{prop}\label{prop4}
The following are equivalent:
\begin{enumerate}
	\item $X\in\Sigma^{r}(H)$;
	\item $X\in\Sigma_{r}(H)$;
\end{enumerate}
\end{prop}
\begin{proof}
Let $X=m+v+A$ be a process of the class $\Sigma_{r}(H)$. That is, $m$ is a local martingale, $v$ and $A$ are processes with finite variations such that $dv$ and $dA$ are respectively carried by $H$ and $\{t\geq0:X_{t}=0\}$. Thus, let $\Gamma=v+A$. We have $\forall t\geq0$,
$$\int{1_{H^{c}}(s)1_{\{X_{s}\neq0\}}d\Gamma_{s}}=\int{1_{\{X_{s}\neq0\}}1_{H^{c}}(s)dv_{s}}+\int{1_{H^{c}}(s)1_{\{X_{s}\neq0\}}dA_{s}}.$$
But, $1_{H^{c}}(s)dv_{s}=0$ and $1_{\{X_{s}\neq0\}}dA_{s}=0$ since $dv$ and $dA$ are respectively carried by $H$ and $\{t\geq0:X_{t}=0\}$. Hence,
$$\int{1_{H^{c}}(s)1_{\{X_{s}\neq0\}}d\Gamma_{s}}=0.$$
Then, $X\in\Sigma_{r}(H)$.

On another hand, consider a process $X=M+\Gamma$ of the class $\Sigma_{r}(H)$. That is, $M\in\mathcal{M}(H)$ and $\Gamma$ is a finite variation process such that
$$\int{1_{H^{c}\cap\{X_{s}\neq0\}}d\Gamma_{s}}=0.$$
Firstly, $M$ decomposes as $M=m+v^{'}$ where $m$ is a local martingale and $v^{'}$ is a finite variation process such $dv^{'}$ is carried by $H$. In addition, we have $\forall t\geq0$,
$$\Gamma_{t}=v^{''}_{t}+A_{t},$$ 
where 
$$v^{''}_{t}=\int_{0}^{t}{1_{\{X_{s}\neq0\}}d\Gamma_{s}}\text{ and }A_{t}=\int_{0}^{t}{1_{\{X_{s}=0\}}d\Gamma_{s}}.$$
However, we have:
$$\int_{0}^{t}{1_{H^{c}}(s)dv^{''}_{s}}=\int_{0}^{t}{1_{H^{c}}(s)1_{\{X_{s}\neq0\}}d\Gamma_{s}}=0$$
and
$$\int_{0}^{t}{1_{\{X_{s}\neq0\}}dA_{s}}=\int_{0}^{t}{1_{\{X_{s}\neq0\}}1_{\{X_{s}=0\}}d\Gamma_{s}}=0.$$
Then, $dv^{''}$ and $dA$ are respectively carried by $H$ and $\{t\geq0:X_{t}=0\}$. Which proves that $X\in\Sigma^{r}(H)$. This completes the proof.
\end{proof}

We recall that the class $\Sigma^{r}(H)$ contains classes $(\Sigma)$ and $\mathcal{M}(H)$. In the following, we provide some processes of the class $\Sigma^{r}(H)$ which are not in both last  mentioned classes. We start with examples inspired by those known on the class $(\Sigma)$.
\begin{exples}\label{ex3} 
For any continuous process $M$ of the class $\mathcal{M}(H)$ with $M_{0}=0$, the following hold:
 \begin{itemize}
	 \item $X=|M|=\int_{0}^{\cdot}{sign(M_{s})dM_{s}}+L_{\cdot}^{0}(M)\in\Sigma^{r}(H)$;
	 \item $\forall \alpha,\beta\in[0,1]$, $Y=\alpha M^{+}+\beta M^{-}\in\Sigma^{r}(H)$;
 \end{itemize}
\end{exples}

 In next, we provide other interesting examples.

\begin{exples}\label{ex4} 
If $D$ is a continuous martingale such that $H=\{t\geq0:D_{t}=0\}$. Hence, 
 \begin{itemize}
	 \item For any continuous local martingale $m$ null at zero, $Z=|\max{\{m-D,m+D\}}-|D_{0}||\in\Sigma^{r}(H)$;
	 \item If $D_{0}=0$, hence $\forall \alpha,\beta\in\R$, $Y=|\alpha m+\beta |D||\in\Sigma^{r}(H)$;
	 \item Let $M$ be a continuous process of $\mathcal{M}(H)$ such that $M_{0}=0$ and $k$ be a bonded progressive process $k$. Define $g_{t}=\sup\{s\leq t:M_{s}=0\}$. The process $T$ defined by $T_{t}=k_{g_{t}}M_{t}$ is in $\Sigma^{r}(H)$.
 \end{itemize}
\end{exples}

%%%%%%%%%%%%%%%%%%%%%%%%%%%%%%%%%%%%%%%%%%%%%%%%%%%%%%%%%%%%%%%%%%%%%%%%%%%%%%%%%%%%%%%%%%%%%%%%%%%%%%%%%%%
%\subsection{Some general properties}
\subsection{Some general properties}
	Now, we shall explore some general properties satisfied by processes of the class $\Sigma^{r}(H)$. Therefore, we start by inferring some properties resulting from integration by parts.

\begin{lem}\label{l1}
Let $X=M+A$ and $Y=W+\Gamma$ be two processes of the class $\Sigma^{r}(H)$ with the decomposition given in Definition \ref{d2}. The process 
$$\left(X_{t}Y_{t}-[X,Y]_{t}-\int_{0}^{t}{X_{s}d\Gamma_{s}}-\int_{0}^{t}{Y_{s}dA_{s}}:t\geq0\right)$$
is an element of the class $\mathcal{M}(H)$.
\end{lem}
\begin{proof}
Through integration by parts, we obtain:
$$X_{t}Y_{t}-[X,Y]_{t}=\int_{0}^{t}{X_{s-}dY_{s}}+\int_{0}^{t}{Y_{s-}dX_{s}}.$$
Which implies that
$$X_{t}Y_{t}-[X,Y]_{t}-\int_{0}^{t}{X_{s-}d\Gamma_{s}}-\int_{0}^{t}{Y_{s-}dA_{s}}=\int_{0}^{t}{X_{s-}dW_{s}}+\int_{0}^{t}{Y_{s-}dM_{s}}.$$
But, we have from the continuity of $A$ and $\Gamma$ that
$$\int_{0}^{t}{X_{s-}d\Gamma_{s}}=\int_{0}^{t}{X_{s}d\Gamma_{s}}\text{ and }\int_{0}^{t}{Y_{s-}dA_{s}}=\int_{0}^{t}{Y_{s}dA_{s}}.$$
In addition, remark that $\int_{0}^{\cdot}{X_{s-}dW_{s}}$ and $\int_{0}^{\cdot}{Y_{s-}dM_{s}}$ are processes of the class $\mathcal{M}(H)$. Hence, it entails that
$$\left(X_{t}Y_{t}-[X,Y]_{t}-\int_{0}^{t}{X_{s}d\Gamma_{s}}-\int_{0}^{t}{Y_{s}dA_{s}}:t\geq0\right)\in\mathcal{M}(H).$$
\end{proof}

In next, we derive a series of corollaries of Lemma \ref{l1}. Thus, we begin by those establishing conditions under which $\left(X_{t}Y_{t}-[X,Y]_{t}-\int_{0}^{t}{X_{s}d\Gamma_{s}}-\int_{0}^{t}{Y_{s}dA_{s}}:t\geq0\right)$ is a local martingale.

\begin{coro}\label{c11}
If $X=M+A$ and $Y=W+\Gamma$ are two processes of the class $\Sigma^{r}(H)$ null on $H$. Hence, the process
$$\left(X_{t}Y_{t}-[X,Y]_{t}-\int_{0}^{t}{X_{s}d\Gamma_{s}}-\int_{0}^{t}{Y_{s}dA_{s}}:t\geq0\right)$$
is a local martingale.
\end{coro}
\begin{proof}
Through Lemma \ref{l1}, $\left(X_{t}Y_{t}-[X,Y]_{t}-\int_{0}^{t}{X_{s}d\Gamma_{s}}-\int_{0}^{t}{Y_{s}dA_{s}}:t\geq0\right)$
is an element of the class $\mathcal{M}(H)$. Furthermore, we have:
$$X_{t}Y_{t}-[X,Y]_{t}-\int_{0}^{t}{X_{s}d\Gamma_{s}}-\int_{0}^{t}{Y_{s}dA_{s}}=\int_{0}^{t}{X_{s-}dW_{s}}+\int_{0}^{t}{Y_{s-}dM_{s}}.$$
But, $M$ and $W$ decompose as $M=m+v$ and $W=m^{'}+v^{'}$, where $m$ and $m^{'}$ are càdlàg local martingales and $v$ and $v^{'}$ are continuous processes with finite variations such that $dv$ and $dv^{'}$ are carried by $H$. Hence, we obtain that $\forall t\geq0$,
$$\int_{0}^{t}{Y_{s-}dM_{s}}=\int_{0}^{t}{Y_{s-}dm_{s}}+\int_{0}^{t}{Y_{s-}dv_{s}}\text{ and }\int_{0}^{t}{X_{s-}dW_{s}}=\int_{0}^{t}{X_{s-}dm^{'}_{s}}+\int_{0}^{t}{X_{s-}dv^{'}_{s}}.$$
On another hand, $v$ and $v^{'}$ are continuous. Hence,
$$\int_{0}^{t}{Y_{s-}dv_{s}}=\int_{0}^{t}{Y_{s}dv_{s}}\text{ and }\int_{0}^{t}{X_{s-}dv^{'}_{s}}=\int_{0}^{t}{X_{s}dv^{'}_{s}}.$$
Therefore,
$$\int_{0}^{t}{Y_{s-}dv_{s}}=0\text{ and }\int_{0}^{t}{X_{s-}dv^{'}_{s}}=0$$
since $Y$ and $X$ vanish on $H$ and $dv$ and $dv^{'}$ are carried by $H$. This completes the proof.
\end{proof}

\begin{rem}\label{R2.1}
We retain according the above proof that for any predictable process $h$ and a process $W\in\mathcal{M}(H)$, the processes $\int_{0}^{\cdot}{h_{s-}dW_{s}}$ and  $\int_{0}^{\cdot}{h_{s}dW_{s}}$ are processes of the class $\mathcal{M}(H)$. And they are local martingales when $h$ vanishes on $H$.
\end{rem}

\begin{coro}\label{c12}
If $X=m+A$ is a process of the class $(\Sigma)$ null on $H$. Hence, for every element $Y=W+\Gamma$ of the class $\Sigma^{r}(H)$, the process
$$\left(X_{t}Y_{t}-[X,Y]_{t}-\int_{0}^{t}{X_{s}d\Gamma_{s}}-\int_{0}^{t}{Y_{s}dA_{s}}:t\geq0\right)$$
is a local martingale.
\end{coro}
\begin{proof}
 According what precedes, we have:
$$X_{t}Y_{t}-[X,Y]_{t}-\int_{0}^{t}{X_{s}d\Gamma_{s}}-\int_{0}^{t}{Y_{s}dA_{s}}=\int_{0}^{t}{X_{s-}dW_{s}}+\int_{0}^{t}{Y_{s-}dm_{s}}.$$
But, $m$ is a local martingale. Hence, $\int_{0}^{\cdot}{Y_{s-}dm_{s}}$ is a local martingale. On another hand, we know from Remark \ref{R2.1} that $\int_{0}^{\cdot}{X_{s-}dW_{s}}$ is also a local martingale since $X$ vanishes on $H$ and $W\in\mathcal{M}(H)$. Which completes the proof.
\end{proof}

\begin{coro}\label{c13}
Let $(X_{t}^{1})_{t\geq0},\cdots,(X_{t}^{n})_{t\geq0}$ be processes of the class $\Sigma^{r}(H)$ such that $[ X^{i},X^{j}]=0$ for $i\neq j$. Hence, the following hold:
\begin{enumerate}
	\item $(\Pi_{i=1}^{n}{X^{i}_{t}})_{t\geq0}$ is also of class $\Sigma^{r}(H)$.
	\item If $\forall i\in\{1,\cdots, n\}$, $X^{i}$ vanishes on $H$. Hence, $(\Pi_{i=1}^{n}{X^{i}_{t}})_{t\geq0}$ is a process of the class $(\Sigma)$.
	\item If $\exists l\in\{1,\cdots, n\}$ such that $X^{l}$ is a process of the class $(\Sigma)$ which vanishes on $H$. Hence, $(\Pi_{i=1}^{n}{X^{i}_{t}})_{t\geq0}$ is a process of the class $(\Sigma)$.
\end{enumerate}
\end{coro}
\begin{proof}
\begin{enumerate}
	\item Let us first take $n=2$. Through Lemma \ref{l1}, we obtain that 
	$$\left(X^{1}_{t}X^{2}_{t}-[X^{1},X^{2}]_{t}-\int_{0}^{t}{X^{1}_{s}dA^{2}_{s}}-\int_{0}^{t}{X^{2}_{s}dA^{1}_{s}}:t\geq0\right)$$
	is a process of the class $\mathcal{M}(H)$. That is, $X^{1}X^{2}\in\Sigma^{r}(H)$ since $[X^{1},X^{2}]=0$. Hence, we obtain by induction that for any family $(X_{t}^{1})_{t\geq0},\cdots,(X_{t}^{n})_{t\geq0}$ of the class $\Sigma^{r}(H)$ such that $[ X^{i},X^{j}]=0$ for $i\neq j$, the process $(\Pi_{i=1}^{n}{X^{i}_{t}})_{t\geq0}$ is also of class $\Sigma^{r}(H)$.
	\item We proceed in the same way as 1) by using Corollary \ref{c12} instead of Lemma \ref{l1} to show that $(\Pi_{i=1}^{n}{X^{i}_{t}})_{t\geq0}$ is a process of the class $(\Sigma)$.
	\item Now, we assume that there exists $l\in\{1,\cdots,n\}$ such that $X^{l}$ is a process of the class $(\Sigma)$ that is vanishing on $H$. Remark that $\forall t\geq0$,
	$$\Pi_{i=1}^{n}{X^{i}_{t}}=X_{t}^{l}\times\Pi_{i=1,i\neq l}^{n}{X^{i}_{t}}.$$
	But, we can see from 1) that $\Pi_{i=1,i\neq l}^{n}{X^{i}_{t}}\in\Sigma^{r}(H)$. Hence, we obtain the result by using Corollary \ref{c12}.
\end{enumerate}
\end{proof}

\begin{coro}\label{c14}
Let $X=m+v+A$ be a process of the class $\Sigma^{r}(H)$. Hence, for every locally bounded Borel function $f$, $f(A)X$ is also a process of the class $\Sigma^{r}(H)$ and $f(A)X-\int_{0}^{\cdot}{f(A_{s})dA_{s}}$ is a process of the class $\mathcal{M}(H)$.
\end{coro}
\begin{proof}
We obtain from Lemma \ref{l1} that $f(A)X\in\Sigma^{r}(H)$ since for each element $X=m+v+A$ of the class $\Sigma^{r}(H)$, $f(A)$ is also a process of the class $\Sigma^{r}(H)$. 
\end{proof}

In next lemma, we study the negative and positive parts of processes of the class $\Sigma^{r}(H)$.
\begin{lem}\label{l2}
Let $X=m+v+A$ be a process of the class $\Sigma^{r}(H)$. The following hold:
\begin{enumerate}
	\item $X^{+}-\int_{0}^{\cdot}{1_{\{X_{s}>0\}}dv_{s}}$ is a local submartingale;
	\item if $X$ has no positive jump. Hence, $X^{+}\in\Sigma^{r}(H)$;
	\item if $X$ has no negative jump. Hence, $X^{-}\in\Sigma^{r}(H)$.
\end{enumerate}
\end{lem}
\begin{proof}
\begin{enumerate}
	\item From Tanaka's formula, we have
	$$X_{t}^{+}=\int_{0}^{t}{1_{\{X_{s-}>0\}}dX_{s}}+\sum_{0<s\leq t}{1_{\{X_{s-}\leq0\}}X_{s}^{+}}+\sum_{0<s\leq t}{1_{\{X_{s-}>0\}}X_{s}^{-}}+\frac{1}{2}L_{t}^{0}.$$
	However,
	$$\hspace{-4cm}\int_{0}^{t}{1_{\{X_{s-}>0\}}dX_{s}}=\int_{0}^{t}{1_{\{X_{s-}>0\}}d(m_{s}+v_{s})}+\int_{0}^{t}{1_{\{X_{s-}>0\}}dA_{s}}$$
	$$\hspace{2cm}=\int_{0}^{t}{1_{\{X_{s-}>0\}}d(m_{s}+v_{s})}+\int_{0}^{t}{1_{\{X_{s}>0\}}dA_{s}}=\int_{0}^{t}{1_{\{X_{s-}>0\}}d(m_{s}+v_{s})}$$
	because $A$ is continuous and $dA$ is carried by $\{t\geq0:X_{t}=0\}$. Then,
	$$\hspace{-4cm}\int_{0}^{t}{1_{\{X_{s-}>0\}}dX_{s}}=\int_{0}^{t}{1_{\{X_{s-}>0\}}dm_{s}}+\int_{0}^{t}{1_{\{X_{s}>0\}}dv_{s}}$$
	because $dv$ is also continuous. Hence, we get
	$$X_{t}^{+}-\int_{0}^{t}{1_{\{X_{s}>0\}}dv_{s}}=\int_{0}^{t}{1_{\{X_{s-}>0\}}dm_{s}}+\sum_{0<s\leq t}{1_{\{X_{s-}\leq0\}}X_{s}^{+}}+\sum_{0<s\leq t}{1_{\{X_{s-}>0\}}X_{s}^{-}}+\frac{1}{2}L_{t}^{0}.$$
 Moreover, $\left(\sum_{0<s\leq t}{1_{\{X_{s-}\leq0\}}X_{s}^{+}}+\sum_{0<s\leq t}{1_{\{X_{s-}>0\}}X_{s}^{-}}+\frac{1}{2}L_{t}^{0};t\geq0\right)$ is an increasing process that is vanishing at zero. Then, $X^{+}-\int_{0}^{\cdot}{1_{\{X_{s}>0\}}dv_{s}}$ is a submartingale, since $m$ and $\int_{0}^{\cdot}{1_{\{X_{s-}>0\}}dm_{s}}$ are local martingales.
	\item We have
	\begin{equation}\label{*}
		X_{t}^{+}-\int_{0}^{t}{1_{\{X_{s}>0\}}dv_{s}}=\int_{0}^{t}{1_{\{X_{s-}>0\}}dm_{s}}+\sum_{0<s\leq t}{1_{\{X_{s-}>0\}}X_{s}^{-}}+\frac{1}{2}L_{t}^{0}
	\end{equation}
	since $X$ has no positive jump. Now, let us set $Z_{t}=\sum_{0<s\leq t}{1_{\{X_{s-}>0\}}X_{s}^{-}}$. Since $m$ and $\int_{0}^{\cdot}{1_{\{X_{s-}>0\}}dm_{s}}$ are local martingales and $v+A$ is continuous, there exists a sequence of stopping times $(T_{n};n\in\N)$ increasing to $\infty$ such that 
	$$E[(X_{T_{n}})^{+}]=E[(m_{T_{n}}+v_{T_{n}}+A_{T_{n}})^{+}]<\infty\text{ and }E\left[\int_{0}^{T_{n}}{1_{\{X_{s-}>0\}}dm_{s}}\right]=0\text{, }n\in\N.$$
	It follows from \eqref{*} that $E[Z_{T_{n}}]\leq E\left[(X_{T_{n}})^{+}-\int_{0}^{T_{n}}{1_{\{X_{s}>0\}}dv_{s}}\right]<\infty$ for all $n\in\N$. Thus, by Theorem VI.80 of \cite{pot}, there exists a right continuous increasing predictable process $V^{Z}$ such that $Z-V^{Z}$ is a local martingale vanishing at zero. Moreover, there exists a sequence of stopping times $(R_{n};n\in\N)$ increasing to $\infty$ such that
	$$E\left[\int_{0}^{t\wedge R_{n}}{1_{\{X^{+}_{s}\neq0\}}dV_{s}^{Z}}\right]=E\left[\int_{0}^{t\wedge R_{n}}{1_{\{X^{+}_{s}\neq0\}}d(V_{s}^{Z}-Z_{s})}+\int_{0}^{t\wedge R_{n}}{1_{\{X^{+}_{s}\neq0\}}dZ_{s}}\right]$$
	$$\hspace{-0.5cm}=E\left[\int_{0}^{t\wedge R_{n}}{1_{\{X^{+}_{s}\neq0\}}dZ_{s}}\right].$$
	Hence,
	$$E\left[\int_{0}^{t\wedge R_{n}}{1_{\{X^{+}_{s}\neq0\}}dV_{s}^{Z}}\right]=E\left[\sum_{0<s\leq t\wedge R_{n}}{1_{\{X_{s}^{+}\neq0\}}1_{\{X_{s-}>0\}}X_{s}^{-}}\right]=E\left[\sum_{0<s\leq t\wedge R_{n}}{1_{\{X_{s}>0\}}1_{\{X_{s-}>0\}}X_{s}^{-}}\right].$$
	Thus,
	$$E\left[\int_{0}^{t\wedge R_{n}}{1_{\{X^{+}_{s}\neq0\}}dV_{s}^{Z}}\right]=0,$$
	since $1_{\{X_{s}>0\}}X_{s}^{-}=0$. This implies that $\int_{0}^{t}{1_{\{X_{s}^{+}\neq0\}}dV_{s}^{Z}}=0$. Therefore, $dV_{t}^{Z}$ is carried by $\{t\geq0;X_{t}^{+}=0\}$. Consequently,
	$$X^{+}_{t}=\left(\int_{0}^{t}{1_{\{X_{s-}>0\}}dm_{s}}+(Z_{t}-V_{t}^{Z})+\int_{0}^{t}{1_{\{X_{s}>0\}}dv_{s}}\right)+\left(V_{t}^{Z}+\frac{1}{2}L_{t}^{0}\right)$$
	is a stochastic process of the class $\Sigma^{r}(H)$.
	\item It is obvious that $(-X)$ is of the class $\Sigma^{r}(H)$ and it has no positive jump. Then, from 3), $X-=(-X)^{+}$ is also of the class $\Sigma^{r}(H)$.
\end{enumerate}
\end{proof}

\begin{rem}\label{r2.1}
For any non negative process $X=m+v+A$ of the class $\Sigma^{r}(H)$, $X-v$ is a local sub-martingale. That is, $A$ is a non decreasing process.
\end{rem}

\begin{coro}
If $X$ is non-negative process of the class $\Sigma^{r}(H)$, hence, its decomposition of the form: $X=m+v+A$ is unique.
\end{coro}
\begin{proof}
By definition, $X=m+v+A$, where $m$ is a local martingale, $v$ and $A$ are processes with finite variations such that $dv$ and $dA$ are carried respectively  by $H$ and $\{t\geq0:X_{t}=0\}$. But, we know from Remark \ref{r2.1} that $X-v$ is a local sub-martingale. Hence, $m$ and $A$ are respectively the unique local martingale and the unique increasing process such that $X-v=m+A$. This entails also the uniqueness of $v$. Which completes the proof.
\end{proof}

Now let's infer the properties deriving from balayage formulas.
\begin{lem}\label{l8}
Let $X=M+A$ be a process of the class $\Sigma^{r}(H)$, and let $g_{t}=\sup\{s\leq t:X_{s}=0\}$. The following hold:
\begin{enumerate}
	\item If $X$ is continuous. Hence, for any locally bounded predictable process $k$, $k_{g_{\cdot}}X$ is an element of the class $\Sigma^{r}(H)$ whose finite variations part is $\Gamma=\int_{0}^{\cdot}{k_{g_{s}}dA_{s}}$;
	\item If $X$ is right continuous. Hence, for any bounded predictable process $k$, $k_{g_{\cdot}}X$ is an element of the class $\Sigma^{r}(H)$ whose finite variations part is $\Gamma=\int_{0}^{\cdot}{k_{g_{s}}dA_{s}}$.
\end{enumerate}
\end{lem}
\begin{proof}
According to Definition \ref{d2}, $X$ decomposes as $X=M+A$, where $M\in\mathcal{M}(H)$ and $A$ is a continuous process with finite variations. 
\begin{enumerate}
	\item By applying the balayage's formula in predictable case for continuous semi-martingales, we obtain the following:
$$k_{g_{t}}X_{t}=\int_{0}^{t}{k_{g_{s}}dX_{s}}=\int_{0}^{t}{k_{g_{s}}dM_{s}}+\int_{0}^{t}{k_{g_{s}}dA_{s}}.$$
But, we know from Remark \ref{R2.1} that $\int_{0}^{\cdot}{k_{g_{s}}dM_{s}}\in\mathcal{M}(H)$. In addition, it is obvious to see that $\Gamma=\int_{0}^{\cdot}{k_{g_{s}}dA_{s}}$ is a continuous process with finite variations such that $d\Gamma$ is carried by $\{t\geq0:k_{g_{t}}X_{t}=0\}$. Consequently, $k_{g_{\cdot}}X\in\Sigma^{r}(H)$.
  \item We obtain the result by proceeding of the same thing by applying the balayage's formula in predictable case for right continuous semimartingales.
\end{enumerate}
\end{proof}

\begin{lem}\label{l9}
Let $X=M+A$ be a continuous process of the class $\Sigma^{r}(H)$, and let $g_{t}=\sup\{s\leq t:X_{s}=0\}$. Hence, the following hold:
\begin{enumerate}
	\item for any bounded progressive process $k$, $k_{g_{\cdot}}X$ is an element of the class $\Sigma^{r}(H)$;
	\item for any càdlàg bounded progressive process $k$ which vanishes on $H$, $k_{g_{\cdot}}X$ is an element of the class $(\Sigma)$.
\end{enumerate}
\end{lem}
\begin{proof}
By applying the balayage's formula for the progressive case, we obtain the following:
$$k_{g_{t}}X_{t}=\int_{0}^{t}{^{p}(k_{g_{s}})dX_{s}}+R_{t},$$
where ${^{p}(k_{g_{s}})}$ is the predictable projection of $k_{g_{s}}$ and $R$ is a continuous process with finite variations such $dR$ is carried by $\{t\geq0:X_{t}=0\}$ and $R_{0}=0$. This implies that:
$$k_{g_{t}}X_{t}=\int_{0}^{t}{^{p}(k_{g_{s}})dM_{s}}+\int_{0}^{t}{^{p}(k_{g_{s}})dA_{s}}+R_{t}.$$

\begin{enumerate}
	\item It can be seen at this level that $k_{g_{\cdot}}X$ is an element of the class $\Sigma^{r}(H)$.
	\item Since $k$ is càdlàg, we have $^{p}(k_{g_{s}})=k_{s-}$. Hence, we obtain:
	$$k_{g_{t}}X_{t}=\int_{0}^{t}{k_{s-}dM_{s}}+\int_{0}^{t}{k_{s-}dA_{s}}+R_{t}.$$
	Which gives:
	$$k_{g_{t}}X_{t}=\int_{0}^{t}{k_{s}dM_{s}}+\int_{0}^{t}{k_{s}dA_{s}}+R_{t}$$
	because $M$ and $A$ are continuous processes. However, $\int_{0}^{\cdot}{k_{s}dM_{s}}$ is a local martingale since $k$ vanishes on $H$. This completes the proof.
\end{enumerate}
\end{proof}	

\begin{lem}\label{l10}
Let $X=M+A$ and $Y=M^{'}+A^{'}$ be two none negative and continuous processes of the class $\Sigma^{r}(H)$ such $X=Y$ on $H$. Hence, the process  $Z=\min\{X,Y\}$ is also an element of the class $\Sigma^{r}(H)$.
\end{lem}
\begin{proof}
We have:
$$\hspace{-7.75cm}2Z_{t}=X_{t}+Y_{t}-|X_{t}-Y_{t}|$$
$$=\left(M_{t}+M^{'}_{t}-\int_{0}^{t}{sign(X_{s}-Y_{s})d(M_{s}-M^{'}_{s})}-L_{t}^{0}(X-Y)\right)+\Gamma_{t},$$
where $\Gamma_{t}=\int_{0}^{t}{[1-sign(X_{s}-Y_{s})]dA_{s}}+\int_{0}^{t}{[1+sign(X_{s}-Y_{s})]dA^{'}_{s}}$. Remark that $dL_{\cdot}^{0}(X-Y)$ is carried by $H$ because $X=Y$ on $H$. In addition, we can see from Remark \ref{R2.1} that
$$\left(M_{t}+M^{'}_{t}-\int_{0}^{t}{sign(X_{s}-Y_{s})d(M_{s}-M^{'}_{s})}-L_{t}^{0}(X-Y)\right)_{t\geq0}\in\mathcal{M}(H).$$
On another hand, we have $\{t\geq0:Z_{t}=0\}=\{t\geq0:X_{t}=0\}\cup\{t\geq0:Y_{t}=0\}$. Hence,
$$\int{1_{\{Z_{s}\neq0\}}dA_{s}}=\int{1_{\{Y_{s}\neq0\}}1_{\{X_{s}\neq0\}}dA_{s}}=0$$
since, $1_{\{X_{s}\neq0\}}dA_{s}=0$. In the same way, we prove that 
$$\int{1_{\{Z_{s}\neq0\}}dA^{'}_{s}}=0.$$
Consequently, we deduce that $d\Gamma$ is carried by $\{t\geq0:Z_{t}=0\}$. This concludes the proof.
\end{proof}

%%%%%%%%%%%%%%%%%%%%%%%%%%%%%%%%%%%%%%%%%%%%%%%%%%%%%%%%%%%%%%%%%%%%%%%%%%%%%%%%%%%%%%%%%%%%%%%%%%%%%%%%%%%%%%%%%%%%%%%%%%%%%%%%%%%%%%%%%%%%%%%%%%%%%%%%%%%%%%%%%%%%%
\section{Characterization results}\label{sec:3}
Now, we shall derive a series of results characterizing processes of the class $\Sigma^{r}(H)$. We start by results which characterize continuous processes of this class.

\subsection{Characterization results for continuous processes} 
\begin{theorem}\label{th1}
Let $X$ be a continuous process. Then, $$X\in \Sigma^{r}(H)\Leftrightarrow |X|\in \Sigma^{r}(H).$$
\end{theorem}
\begin{proof}
$\Rightarrow)$ Let $X=m+v+A$ be a process of the class $\Sigma^{r}(H)$. We obtain from Tanaka's formulas that: 
$$\hspace{-4.5cm}|X_{t}|=\int^{t}_{0}{Sgn(X_{s})dX_{s}}+L^{0}_{t}(X)$$ 
$$\hspace{2cm}=\int^{t}_{0}{Sgn(X_{s})dm_{s}}+\int^{t}_{0}{Sgn(X_{s})dv_{s}}+\int^{t}_{0}{Sgn(X_{s})dA_{s}}+L_{t}^{0}(X).$$ 
Hence, 
$$|X_{t}|=\int^{t}_{0}{Sgn(X_{s})dm_{s}}+\int^{t}_{0}{Sgn(X_{s})dv_{s}}+L_{t}^{0}(X)$$
since $\int^{t}_{0}{Sgn(X_{s})dA_{s}}=0$ because $Supp(dA)\subset \{X=0\}$. Thus, $|X|\in \Sigma^{r}(H)$.\\ 
$\Leftarrow)$ Now, we assume that $|X|\in \Sigma^{r}(H)$. This means that $|X|=m+v+A$ where $m$ is a local martingale and $v$ and $A$ are continuous finite variation processes such $dv$ and $dA$ are carried respectively by $H$ and $\{t\geq0:X_{t}=0\}$. Let us put 
$$k_{t}=\liminf_{s\rightarrow t}\left\{1_{X_{s}>0}-1_{X_{s}<0}\right\} \text{ and } g_{t}=\sup\left\{s\leq t: X_{s}=0\right\}.$$ 
It comes from Balayage's formula that: 
$$k_{g_{t}}|X_{t}|=\int^{t}_{0}{^{p}(k_{g_{s}})d|X_{s}|}+R_{t}$$
where $R$is an increasing and continuous process such as $Supp(dR_{t})\subset \left\{|X_{t}|=0\right\}=\left\{X_{t}=0\right\}$. But, 
$$\int^{t}_{0}{^{p}(k_{g_{s}})d|X_{s}|}=\int^{t}_{0}{k_{s-}d|X_{s}|}=\int^{t}_{0}{k_{s}d|X_{s}|}$$
since $|X|$ is continuous and $k$ is a càdlàg progressive process. Hence, 
$$X_{t}=\int^{t}_{0}{k_{s}dm_{s}}+\int^{t}_{0}{k_{s}dv_{s}}+\left[\int^{t}_{0}{k_{s}dA_{s}}+R_{t}\right].$$ 
However, 
$$\int^{t}_{0}{k_{s}dA_{s}}=0 \text{ because } Supp(dA)\subset \left\{t\geq 0: X_{t}=0\right\}.$$
This completes the proof.
\end{proof}

\begin{theorem}\label{th2}
A continuous semi-martingale $X$ is an element of the class $\Sigma^{r}(H)$ if, and only if, there exists a process $W$ of the class $\mathcal{M}(H)$ such that 
$|X|=|W|$.
\end{theorem}
\begin{proof}
$\Rightarrow)$ Let us assume that $X\in\Sigma^{r}(H)$. That is, there exist $M\in\mathcal{M}(H)$ and a finite variations process $A$ such that $dA$ is carried by $\{t\geq0:X_{t}=0\}$ and $X=M+A$. Now, let $Z^{\alpha}$ be the process defined in \eqref{zalpha} and constructed with respect to $X$. Hence, we obtain from Proposition \ref{pzalph} that:
$$Z^{\alpha}_{t}X_{t}=\int_{0}^{t}{Z^{\alpha}_{s}dX_{s}}+(2\alpha-1)L_{t}^{0}(Z^{\alpha}X)$$
$$\hspace{3.25cm}=\int_{0}^{t}{Z^{\alpha}_{s}dM_{s}}+\int_{0}^{t}{Z^{\alpha}_{s}dA_{s}}+(2\alpha-1)L_{t}^{0}(Z^{\alpha}X)$$
$$\hspace{1.25cm}=\int_{0}^{t}{Z^{\alpha}_{s}dM_{s}}+(2\alpha-1)L_{t}^{0}(Z^{\alpha}X)$$
since $\int_{0}^{t}{Z^{\alpha}_{s}dA_{s}}=0$. Indeed, $dA$ is carried by $\{t\geq0:X_{t}=0\}$ and $\{Z^{\alpha}_{s}=0\}=\{X_{s}=0\}$. In particular case where $\alpha=\frac{1}{2}$, we obtain:
$$Z^{\alpha}_{t}X_{t}=\int_{0}^{t}{Z^{\alpha}_{s}dM_{s}}.$$
But, we have from Remark \ref{R2.1} that $W=\int_{0}^{\cdot}{Z^{\alpha}_{s}dM_{s}}\in\mathcal{M}(H)$. Furthermore, $Z^{\alpha}_{t}\in\{-1,0,1\}$ and $\{Z^{\alpha}_{s}=0\}=\{X_{s}=0\}$. Consequently, $|X|=|W|$.\\
$\Leftarrow)$ Now, consider that there exists a process $W$ of the class $\mathcal{M}(H)$ such that $|X|=|W|$. Hence, we have according ton Example \ref{ex3} that $|X|\in\Sigma^{r}(H)$. Consequently, we obtain from Theorem \ref{th1} that $X\in\Sigma^{r}(H)$. Which completes the proof.
\end{proof}

\begin{rem}\label{rem3.1}
According to Theorem \ref{th2}, we know now that for any positive process $X$ of the class  $\Sigma^{r}(H)$, there exists a process $W$ of the class $\mathcal{M}(H)$ such that $X=|W|$. In addition, we have also obtained in the above proof that $\forall t\geq0$, $W_{t}=\int_{0}^{t}{Z^{\alpha}_{s}dX_{s}}$ with $\alpha=0,5$.
\end{rem}

Bouhadou and Ouknine have shown in Proposition 2.3 of \cite{siam}, a similar result for processes of the class $(\Sigma)$. We obtain this result in next.
\begin{coro}
If $X$ is a continuous and positive process of the class $(\Sigma)$, then there exists a continuous martingale $W$ such that $X=|W|$.
\end{coro}
\begin{proof}
Indeed we have from Theorem \ref{th2} and Remark \ref{rem3.1} that $X=|W|$, where $\forall t\geq0$, $W_{t}=\int_{0}^{t}{Z^{\alpha}_{s}dX_{s}}$ with $\alpha=0,5$. However, $X\in(\Sigma)$. Then, it decomposes as $X=m+A$ where $m$ is a local martingale and $A$ is a continuous finite variation process such that $dA$ is carried by $\{t\geq0:X_{t}=0\}$. Hence, we get that $W_{t}=\int_{0}^{t}{Z^{\alpha}_{s}dm_{s}}$ since $\{t\geq0:X_{t}=0\}=\{t\geq0:Z^{\alpha}_{t}=0\}$. This concludes the proof.
\end{proof}
%%%%%%%%%%%%%%%%%%%%%%%%%%%%%%%%%%%%%%%%%%%%%%%%%%%%%%%%%%%%%%%%%%%%%%%%%%%%%%%%%%%%%%%
\subsection{Characterization results for càdlàg processes} 

Now, we shall derive characterize results for càdlàg stochastic processes of the class $\Sigma^{r}(H)$. Thus, we start with processes whose the finite variational part is considered continuous.
\begin{theorem}\label{th3}
Let $X$ be a càdlàg non-negative semi-martingale. The following are equivalent:
\begin{enumerate}
	\item $X\in \Sigma^{r}(H)$; 	
	\item There exists a non-decreasing continuous process $V$ such that for any locally bounded Borel function $f:\R_{+}\rightarrow\R_{+}$ and defining $F(x)=\int_{0}^{x}{f(y)dy}$, the process $\left(f(V_{t})X_{t}-F(V_{t})\right)_{t\geq 0}$ is an element of the class $\mathcal{M}(H)$. 
\end{enumerate}
\end{theorem}

\begin{proof}

$(1)\Rightarrow (2)$ Let $X$ be a process of the class $\Sigma^{r}(H)$. That is, it decomposes as $X=M+A$, where $M\in\mathcal{M}(H)$ and $A$ is a continuous finite variations process such that $dA$ is carried  by $\{t\geq0:X_{t}=0\}$. In fact, we have from Remark \ref{r2.1} that $A$ is non-decreasing. Thus, we put $V=A$. We obtain from balayage's formula in the predictable case that for any locally bounded Borel function $f:\R_{+}\rightarrow\R_{+}$, $\forall t\geq0$,
$$f(A_{t})X_{t}=\int_{0}^{t}{f(A_{s})dX_{s}}=\int_{0}^{t}{f(A_{s})dM_{s}}+F(A_{t}).$$
That is,
$$f(A_{t})X_{t}-F(A_{t})=\int_{0}^{t}{f(A_{s})dM_{s}}.$$
Hence, this entails from Remark \ref{R2.1} that $\left(f(V_{t})X_{t}-F(V_{t})\right)_{t\geq 0}$ is an element of the class $\mathcal{M}(H)$.
\\
$(2)\Rightarrow (1)$ Let us first take $F(x)=x$. Thus, $W=V-X\in\mathcal{M}(H)$. Now, for $F(x)=x^{2}$, we obtain that $W^{'}=V^{2}-2VX\in\mathcal{M}(H)$. An application of integration by parts entails that $\forall t\geq0$,
$$W^{'}_{t}=2\int_{0}^{t}{V_{s}dW_{s}}-2\int_{0}^{t}{X_{s}dV_{s}}.$$
Hence, $\int_{0}^{\cdot}{X_{s}dV_{s}}\in\mathcal{M}(H)$. That is, 
$$\int{1_{\{D_{s}\neq0\}}X_{s}dV_{s}}=0.$$
This implies that
$$\int{1_{\{D_{s}X_{s}\neq0\}}dV_{s}}=0.$$
Consequently, $dV$ is carried by $H\cup\{t\geq0:X_{t}=0\}$. Hence, we obtain from Proposition \ref{prop4} that $X\in\Sigma^{r}(H)$. Which completes the demonstration.
\end{proof}

\begin{rem}
We have proved in the above proof that $\forall t\geq0$,
$$f(A_{t})X_{t}-F(A_{t})=\int_{0}^{t}{f(A_{s})dM_{s}}.$$
Hence, if $X\in(\Sigma)$. $f(A)X-F(A)$ is a local martingale. Which permits us to recover an interesting characterization result for processes of the class $(\Sigma)$ that Nikeghbali called characterization martingale (Theorem 2.1 of \cite{nik}). 
\end{rem}

Now, we shall extend the notion of class $\Sigma^{r}(H)$ to càdlàg special semi-martingales $X=m+v+A$ whose the finite variational part $A$ is càdlàg instead of continuous.
\begin{theorem}\label{th4}
let $X$ be a càdlàg semi-martingale. The following are equivalents: 
\begin{enumerate}
	\item $X\in \Sigma^{r}(H)$; 		
	\item There exists a càdlàg predictable process $V$ with finite variations such that for any function $f\in\mathcal{C}^1(\mathbb{R})$ and defining $F(x)=\int_{0}^{x}{f(y)dy}$, the process 
$$\left(f(V^{c}_{t})X_{t}-F(V^{c}_{t})-\sum_{0<s\leq t}{[f(V^{c}_{s})-f^{'}(V^{c}_{s})X_{s}]\Delta V_{s}}\right)_{t\geq 0}$$
 is an element of the class $\mathcal{M}(H)$. 
\end{enumerate}
\end{theorem}

\begin{proof}
$(1)\Rightarrow(2)$ Let us consider $V=A$. Hence, through integration by parts, we get
$$f(A^{c}_{t})X_{t}=\int_{0}^{t}{f(A^{c}_{s})dX_{s}}+\int_{0}^{t}{f^{'}(A^{c}_{s})X_{s}dA^{c}_{s}}.$$
Hence, we have  
$$f(A^{c}_{t})X_{t}=\int_{0}^{t}{f(A^{c}_{s})dX_{s}}+\int_{0}^{t}{f^{'}(A^{c}_{s})X_{s}dA_{s}}-\sum_{s\leq t}{f^{'}(A^{c}_{t})X_{s}\Delta A_{s}}$$
because $A=A^{c}+\sum_{s\leq t}{\Delta A_{s}}$. Furthermore, we have $\int_{0}^{t}{f^{'}(A^{c}_{s})X_{s}dA_{s}}=0$ since $dA$ is carried by $\{t\geq0:X_{t}=0\}$. Therefore, it follows that
$$f(A^{c}_{t})X_{t}=\int_{0}^{t}{f(A^{c}_{s})dX_{s}}-\sum_{s\leq t}{f^{'}(A^{c}_{t})X_{s}\Delta A_{s}}$$
$$\hspace{2cm}=\int_{0}^{t}{f(A^{c}_{s})dM_{s}}+\int_{0}^{t}{f(A^{c}_{s})dA^{c}_{s}}+\sum_{s\leq t}{[f(A^{c}_{s})-f^{'}(A^{c}_{s})X_{s}]\Delta A_{s}}.$$
Consequently,
$$f(A^{c}_{t})X_{t}=\int_{0}^{t}{f(A^{c}_{s})dM_{s}}+F(A^{c}_{t})+\sum_{s\leq t}{[f(A^{c}_{s})-f^{'}(A^{c}_{s})X_{s}]\Delta A_{s}}.$$
This implies that
$$F(A^{c}_{t})+\sum_{s\leq t}{[f(A^{c}_{s})-f^{'}(A^{c}_{s})X_{s}]\Delta A_{s}}-f(A^{c}_{t})X_{t}=-\int_{0}^{t}{f(A^{c}_{s})dM_{s}}.$$
Hence, we deduce from Remark \ref{R2.1} that 
$$\left(f(V^{c}_{t})X_{t}-F(V^{c}_{t})-\sum_{0<s\leq t}{[f(V^{c}_{s})-f^{'}(V^{c}_{s})X_{s}]\Delta V_{s}}\right)_{t\geq 0}$$
 is an element of the class $\mathcal{M}(H)$.\\
$(2)\Rightarrow(1)$ First, let $F(x)=x$. Then, the process $W$ defined by 
$$W_{t}=V^{c}_{t}+\sum_{s\leq t}{\Delta V_{s}}-X_{t}=V_{t}-X_{t}$$
is an element of the class $\mathcal{M}(H)$. Next, we take $F(x)=x^{2}$. Thus, process $B$ defined by 
$$B_{t}=(V_{t}^{c})^{2}-2V_{t}^{c}X_{t}+2\sum_{s\leq t}{V_{s}^{c}\Delta V_{s}}-2\sum_{s\leq t}{X_{s}\Delta V_{s}}$$
is a process of the class $\mathcal{M}(H)$. However, through integration by parts, it follows that
$$B_{t}=2\int_{0}^{t}{V^{c}_{s}dV^{c}_{s}}-2\int_{0}^{t}{V^{c}_{s}dX_{s}}-2\int_{0}^{t}{X_{s}dV^{c}_{s}}+2\sum_{s\leq t}{V_{s}^{c}\Delta V_{s}}-2\sum_{s\leq t}{X_{s}\Delta V_{s}}$$
$$\hspace{-0.75cm}=2\int_{0}^{t}{V^{c}_{s}d\left(V^{c}_{s}+\sum_{u\leq s}{\Delta V_{u}}-X_{s}\right)}-2\int_{0}^{t}{X_{s}d\left(V^{c}_{s}+\sum_{u\leq s}{\Delta V_{u}}\right)}$$
$$\hspace{-6cm}=2\int_{0}^{t}{V^{c}_{s}dW_{s}}-2\int_{0}^{t}{X_{s}dV_{s}}.$$
Consequently, we must have that $\int_{0}^{\cdot}{X_{s}dV_{s}}\in\mathcal{M}(H)$. That is,
$$\int_{0}^{t}{1_{\{D_{s}\neq0\}}X_{s}dV_{s}}=0.$$
Which also means that
$$\int_{0}^{t}{1_{\{D_{s}X_{s}\neq0\}}dV_{s}}=0.$$
In other words, $dV$ is carried  by the set $\{t\geq0:D_{s}X_{t}=0\}$. Consequently, $X\in\Sigma^{r}(H)$.
\end{proof}

%%%%%%%%%%%%%%%%%%%%%%%%%%%%%%%%%%%%%%%%%%%%%%%%%%%%%%%%%%%%%%%%%%%%%%%
\section{Representation results with respect to the honest times \texorpdfstring{$\gamma$}{gama} and \texorpdfstring{$g$}{g}}\label{sec:4}

In this section, we provide some formulas permitting to represent some processes of the class $\Sigma^{r}(H)$ using honest times  $\gamma$ and $g$. These formulas are inspired by a representation result given for relative martingales in Proposition 2.2 of \cite{1} and of an another representation formula given in Theorem 3.1 of \cite{pat} for processes of the class $(\Sigma)$. More precisely, for all stopping time $T<\infty$, any process $X$ of the class $\mathcal{R}(H)$ takes the form: $X_{T}=\E[X_{\infty}1_{\{\gamma<T\}}|\mathcal{F}_{T}]$. Whereas under some assumptions, a process $X$ of class $(\Sigma)$ is written as follows:  $X_{T}=\E[X_{\infty}1_{\{g<T\}}|\mathcal{F}_{T}]$. 

Throughout this section, we consider $\gamma$ and $g$ such that  $\P(\gamma<\infty)=1$ and $\P(g<\infty)=1$.

\begin{prop}\label{p2}
Let $X=m+v+A$ be a process of the class $\Sigma^{r}(H)$ such that $m$ is a true martingale and $\lim_{t\to+\infty}{X_{t}}$ exists. Let $\lim_{t\to+\infty}{X_{t}}=X_{\infty}$, $g=\sup\{t\geq0:X_{t}=0\}$ and $d_{t}=\inf\{s>t\geq0:X_{s}=0\}$. Hence, for all stopping time $T<\infty$, we have:
\begin{equation}
	X_{T}=\E\left[X_{\infty}1_{\{g<T\}}|\mathcal{F}_{T}\right]+\E\left[(v_{T}-v_{d_{T}})1_{\{\gamma>T\}}|\mathcal{F}_{T}\right].
\end{equation}
\end{prop}
\begin{proof}
Firstly, remark that we have: $X_{\infty}1_{\{g<T\}}=X_{d_{T}}=m_{d_{T}}+v_{d_{T}}+A_{d_{T}}$. Since $dA$ is carried by $\{t\geq0:X_{t}=0\}$ and $g<T$, $A_{T}=A_{d_{T}}$. Thus, we obtain:
$$\E\left[X_{\infty}1_{\{g<T\}}|\mathcal{F}_{T}\right]=\E\left[m_{d_{T}}|\mathcal{F}_{T}\right]+\E\left[v_{d_{T}}|\mathcal{F}_{T}\right]+A_{T}.$$
Hence, it follows that
$$\E\left[X_{\infty}1_{\{g<T\}}|\mathcal{F}_{T}\right]=m_{T}+\E\left[v_{d_{T}}|\mathcal{F}_{T}\right]+A_{T}$$
because $m$ is a true martingale and $T$ and $d_{T}$ are stopping time such that $T<d_{T}$. Therefore, this entails that
$$\E\left[X_{\infty}1_{\{g<T\}}|\mathcal{F}_{T}\right]=X_{T}+\E\left[(v_{d_{T}}-v_{T}  )|\mathcal{F}_{T}\right].$$
On another hand, we have:
$$\E\left[(v_{d_{T}}-v_{T}  )|\mathcal{F}_{T}\right]=\E\left[(v_{d_{T}}-v_{T}  )1_{\{\gamma<T\}}|\mathcal{F}_{T}\right]+\E\left[(v_{d_{T}}-v_{T}  )1_{\{\gamma>T\}}|\mathcal{F}_{T}\right].$$
However, 
$$\E\left[(v_{d_{T}}-v_{T}  )1_{\{\gamma<T\}}|\mathcal{F}_{T}\right]=0$$
because, $v$ is constant after $\gamma$. Consequently,we obtain:
$$X_{T}=\E\left[X_{\infty}1_{\{g<T\}}|\mathcal{F}_{T}\right]+\E\left[(v_{T}-v_{d_{T}})1_{\{\gamma>T\}}|\mathcal{F}_{T}\right].$$
This completes the proof.
\end{proof}

In next corollary, we obtain the representation formula given in Theorem 3.1 of \cite{pat} for processes of the class $(\Sigma)$.
\begin{coro}
If in addition of assumptions of Proposition \ref{p2}, $X\in(\Sigma)$. Hence, for all stopping time $T<\infty$ we have:
\begin{equation}
	X_{T}=\E\left[X_{\infty}1_{\{g<T\}}|\mathcal{F}_{T}\right].
\end{equation}
\end{coro}
\begin{proof}
Remark that $dv$ is carried by $\{t\geq0:X_{t}=0\}$ since $X\in(\Sigma)$ and  $dA$ is already carried by $\{t\geq0:X_{t}=0\}$. Hence, for $d_{T}=\inf\{s>T:X_{s}=0\}$, we obtain that $v_{d_{T}}=v_{T}$. Which implies the result.
\end{proof}

In the next proposition, we give the representation formula which generalizes Proposition 2.2 of \cite{1}.
\begin{prop}\label{p3}
Let $X=m+v+A$ be a process of the class $\Sigma^{r}(H)$ such that $m$ is a true martingale and $\lim_{t\to+\infty}{X_{t}}$ exists. Let $\lim_{t\to+\infty}{X_{t}}=X_{\infty}$, $\gamma=\sup\{t\geq0:t\in H\}$ and $d^{'}_{t}=\inf\{s>t\geq0:s\in H=0\}$. Hence, for all stopping time $T<\infty$, we have:
\begin{equation}
	X_{T}=\E\left[X_{\infty}1_{\{\gamma<T\}}|\mathcal{F}_{T}\right]+\E\left[(A_{T}-A_{d^{'}_{T}})1_{\{g>T\}}|\mathcal{F}_{T}\right].
\end{equation}
\end{prop}
\begin{proof}
Firstly, remark that we have: $X_{\infty}1_{\{\gamma<T\}}=X_{d^{'}_{T}}=m_{d^{'}_{T}}+v_{d^{'}_{T}}+A_{d^{'}_{T}}$. Since $dv$ is carried by $H$ and $\gamma<T$, $v_{T}=v_{d^{'}_{T}}$. Thus, we obtain:
$$\E\left[X_{\infty}1_{\{\gamma<T\}}|\mathcal{F}_{T}\right]=\E\left[m_{d^{'}_{T}}|\mathcal{F}_{T}\right]+\E\left[A_{d^{'}_{T}}|\mathcal{F}_{T}\right]+v_{T}.$$
Hence, it follows that
$$\E\left[X_{\infty}1_{\{\gamma<T\}}|\mathcal{F}_{T}\right]=m_{T}+\E\left[A_{d^{'}_{T}}|\mathcal{F}_{T}\right]+v_{T}$$
because $m$ is a true martingale and $T$ and $d^{'}_{T}$ are stopping time such that $T<d^{'}_{T}$. Therefore, this entails that
$$\E\left[X_{\infty}1_{\{\gamma<T\}}|\mathcal{F}_{T}\right]=X_{T}+\E\left[(A_{d^{'}_{T}}-A_{T}  )|\mathcal{F}_{T}\right].$$
On another hand, we have:
$$\E\left[(A_{d^{'}_{T}}-A_{T}  )|\mathcal{F}_{T}\right]=\E\left[(A_{d^{'}_{T}}-A_{T}  )1_{\{g<T\}}|\mathcal{F}_{T}\right]+\E\left[(A_{d^{'}_{T}}-A_{T}  )1_{\{g>T\}}|\mathcal{F}_{T}\right].$$
However, 
$$\E\left[(A_{d^{'}_{T}}-A_{T}  )1_{\{g<T\}}|\mathcal{F}_{T}\right]=0$$
because, $A$ is constant after $g$. Consequently,we obtain:
$$X_{T}=\E\left[X_{\infty}1_{\{\gamma<T\}}|\mathcal{F}_{T}\right]+\E\left[(A_{T}-A_{d^{'}_{T}})1_{\{g>T\}}|\mathcal{F}_{T}\right].$$
This completes the proof.
\end{proof}

Now, we shall deduce the representation formula given by Azema and Yor in Proposition 2.2 of \cite{1} for relative martingales.
\begin{coro}
If in addition of assumptions of Proposition \ref{p3}, $X\in\mathcal{M}(H)$. Hence, for all stopping time $T<\infty$ we have:
\begin{equation}
	X_{T}=\E\left[X_{\infty}1_{\{\gamma<T\}}|\mathcal{F}_{T}\right].
\end{equation}
\end{coro}
\begin{proof}
Remark that $dA$ is carried by $H$ since $X\in\mathcal{M}(H)$ and  $dv$ is already carried by $H$. Hence, for $d^{'}_{T}=\inf\{s>T:s\in H\}$, we obtain that $A_{d^{'}_{T}}=A_{T}$. Which implies the result.
\end{proof}

Recall that $\gamma$ and $g$ are honest times. Hence, for all $(\mathcal{F}_{t})_{t\geq0}$ stopping time $T$, we have: $\P(\gamma=T)=0$ and $\P(g=T)=0$. Let $(\mathcal{G}^{\gamma}_{t})_{t\geq0}$ and $(\mathcal{G}^{g}_{t})_{t\geq0}$ be predictable enlarged filtrations respectively with respect to $\gamma$ and $g$. On another hand, it is known that for any honest time $\Gamma$, there exists an optional random closed set $H^{'}$ such that  $\Gamma=\sup\{H^{'}\}$. In the following corollary, we show that processes $X_{\cdot+\gamma}$ and $X_{\cdot+g}$ are relative martingales respectively under filtrations  $(\mathcal{G}^{\gamma}_{t+\gamma})_{t\geq0}$ and $(\mathcal{G}^{g}_{t+g})_{t\geq0}$ with respect to random sets $H^{'}$ and $H^{''}$ satisfying $g-\gamma=\sup\{H^{'}\}$ and $\gamma-g=\sup\{H^{''}\}$.

\begin{coro}\label{c8}
Let $X=m+v+A$ be a process of the class $\Sigma^{r}(H)$ such that $m$ is a true martingale and $\lim_{t\to+\infty}{X_{t}}$ exists. Let $\lim_{t\to+\infty}{X_{t}}=X_{\infty}$, $g=\sup\{t\geq0:X_{t}=0\}$ and  $\gamma=\sup\{t\geq0:t\in H\}$. Hence, for all stopping time $0<T<\infty$, we have:
\begin{equation}
	X_{T+\gamma}=\E\left[X_{\infty}1_{\{g-\gamma<T\}}|\mathcal{G}^{\gamma}_{T+\gamma}\right]
\end{equation}
and 
\begin{equation}
	X_{T+g}=\E\left[X_{\infty}1_{\{\gamma-g<T\}}|\mathcal{G}^{g}_{T+g}\right].
\end{equation}
\end{coro}
\begin{proof}
According to Proposition \ref{p2} and Proposition \ref{p3}, we have:
\begin{equation}
	X_{T}=\E\left[X_{\infty}1_{\{g<T\}}|\mathcal{F}_{T}\right]+\E\left[(v_{T}-v_{d_{T}})1_{\{\gamma>T\}}|\mathcal{F}_{T}\right]
\end{equation}
and 
\begin{equation}
	X_{T}=\E\left[X_{\infty}1_{\{\gamma<T\}}|\mathcal{F}_{T}\right]+\E\left[(A_{T}-A_{d^{'}_{T}})1_{\{g>T\}}|\mathcal{F}_{T}\right].
\end{equation}
Hence, we obtain the following:
$$X_{T+\gamma}=\E\left[X_{\infty}1_{\{g<T+\gamma\}}|\mathcal{F}_{T+\gamma}\right]$$
and
$$X_{T+g}=\E\left[X_{\infty}1_{\{\gamma<T+g\}}|\mathcal{F}_{T+g}\right].$$
But we know from Lemma 5.7 of \cite{jeu} that $\mathcal{F}_{T+\gamma}=\mathcal{G}^{\gamma}_{T+\gamma}$ and $\mathcal{F}_{T+\gamma}=\mathcal{G}^{g}_{T+g}$. Which implies the result.
\end{proof}

\begin{coro}\label{c9}
Let $M$ be a positive process of the class $\mathcal{M}(H)$ which has no negative jump such that $\lim_{t\to+\infty}{M_{t}}=0$. Consider a real $k>0$ and define $g_{k}=\sup\{t\geq0:M_{t}\geq k\}$. Hence, for all stopping time $T<\infty$, we have:
\begin{equation}
	\P\left[g_{k}-\gamma>T|\mathcal{F}_{T+\gamma}\right]=1\wedge\left(\frac{M_{T+\gamma}}{k}\right) 
\end{equation}
in particular
\begin{equation}
	\P\left[g_{k}>\gamma|\mathcal{F}_{\gamma}\right]=1\wedge\left(\frac{M_{\gamma}}{k}\right).
\end{equation}
\end{coro}
\begin{proof}
We have from Lemma \ref{l2} that $X=(k-M)^{+}$ is a process of the class $\Sigma^{r}(H)$. Hence, by applying Corollary \ref{c8}, we get:
$$X_{T+\gamma}=\E\left[X_{\infty}1_{\{g_{k}-\gamma<T\}}|\mathcal{F}_{T+\gamma}\right].$$
That is,
$$(k-M_{T+\gamma})^{+}=k\P\left[g_{k}-\gamma<T|\mathcal{F}_{T+\gamma}\right].$$
Which implies the following:
$$\P\left[g_{k}-\gamma>T|\mathcal{F}_{T+\gamma}\right]=1-\left(1-\frac{M_{T+\gamma}}{k}\right)^{+}.$$
Consequently, we get:
$$\P\left[g_{k}-\gamma>T|\mathcal{F}_{T+\gamma}\right]=1\wedge\left(\frac{M_{T+\gamma}}{k}\right).$$
This completes the proof.
\end{proof}

\begin{coro}
Let $M=m+v$ be a positive process of the class $\mathcal{M}(H)$ such that $m$ is a uniformly integrable martingale which has no negative jump, $\langle M,D\rangle=0$ and $\lim_{t\to+\infty}{M_{t}}=0$. Consider a real $k>0$ and define $g_{k}=\sup\{t\geq0:M_{t}\geq k\}$. Hence, for all stopping time $T<\infty$, we have:
\begin{equation}
	\P\left[g_{k}-\gamma>T\right]=1\wedge\left(\frac{\E[M_{\gamma}]}{k}\right) 
\end{equation}
\end{coro}
\begin{proof}
We obtain from Corollary \ref{c9} that
$$\E\left(\E\left[1_{\{g_{k}-\gamma>T\}}|\mathcal{F}_{T+\gamma}\right]\right)=1\wedge\left(\frac{\E[M_{T+\gamma}]}{k}\right).$$
That is,
$$\P\left[g_{k}-\gamma>T\right]=1\wedge\left(\frac{\E[M_{T+\gamma}]}{k}\right).$$
But, $M_{\cdot+\gamma}$ is a martingale since $M\in\mathcal{M}(H)$. Hence, we get:
$$\P\left[g_{k}-\gamma>T\right]=1\wedge\left(\frac{\E[M_{\gamma}]}{k}\right).$$
\end{proof}

%%%%%%%%%%%%%%%%%%%%%%%%%%%%%%%%%%%%%%%%%%%%%%%%%%%%%%%%%%%%%%%%%%%%%%%%%%%%%%%%%%%%%%%%%%%%%%%%%%%%%%%%%%%%%%%%%%%%%%%%%%%%%%%%%%%%%%%%%%
\section{Interesting utilities of stochastic processes of the class \texorpdfstring{$\Sigma^{r}(H)$}{Sigma(r)}}\label{sec:5}

The purpose of the current section is to show that the stochastic processes studied in this paper could have good applications. For this, we propose to construct solutions for skew Brownian motion equations. More precisely, we construct solutions from continuous processes of the class $\Sigma^{r}(H)$ for the following equations:

\begin{equation}\label{1}
	X_{t}=x+B_{t}+(2\alpha-1)L_{t}^{0}(X)
\end{equation}
and
\begin{equation}\label{2}
	X_{t}=x+B_{t}+\int_{0}^{t}{(2\alpha(s)-1)dL_{s}^{0}(X)},
\end{equation}
where $B$ is a standard Brownian motion and $x=0$. It must be remarked that solutions had already been built from the processes of the class $(\Sigma)$ (see \cite{fjo}). This should not be seen as a redundancy because some processes of the class $\Sigma^{r}(H)$ are not elements of the class $(\Sigma)$. For instance, if there exists a Brownian motion $W$ which is independent of $B$ such that $H=\{t\geq0:W_{t}=0\}$, hence the geometric Itô-Mckean skew Brownian motion $X^{\delta}=\sqrt{1-\delta^{2}}B+\delta|W|$ and its absolute value, $|X^{\delta}|$ are such examples.

%%%%%%%%%%%%%%%%%%%%%%%%%%%%%%%%%%%%%%%%%%%%%%%%%%%%%%%%%%%%%%%%%%%%%%%%%%%%%%%%%%%%%%%%%
\subsection{Construction of solution from It\^o-Mckean skew brownian motion}

First, we use the absolute value of the geometric Itô-Mckean skew Brownian process $|X^{\delta}|=\left|\sqrt{1-\delta^{2}}B+\delta|W|\right|$. It is true that we presented this process above as an element of the class $\Sigma^{r}(H)$. But in reality, it is only when the process $W$ cancels on $H$ that $|X^{\delta}|\in\Sigma^{r}(H)$. Thus, we dissociate the construction of solutions using this process from those using the other processes of the class $\Sigma^{r}(H)$. 

We will therefore consider the following notations: we shall set $k^{W}$ and $Z_{1}$ to represent processes constructed in \eqref{zalpha} with respect to $W$ and $(k^{W}_{g_{t}}X^{\delta}_{t};t\geq0)$ respectively and, $Z_{2,\cdot}$ will be the process defined in \eqref{Zalpha} with respect to $(k^{W}_{g_{t}}X^{\delta}_{t};t\geq0)$. We shall also set $g_{t}=\sup{\{t\geq0:X^{\delta}_{t}=0\}}$.

\begin{prop}
The process $Y^{\delta}_{1,\cdot}$ defined by $\forall t\geq0$, $Y^{\delta}_{1,t}=Z_{1,t}k^{W}_{g_{t}}|X^{\delta}_{t}|$ is a weak solution of \eqref{1} with the parameter $\alpha$ and starting from 0. 
\end{prop}
\begin{proof}
By applying the balayage's formula in the progressive case, we get
$$k^{W}_{g_{t}}|X^{\delta}_{t}|=\int_{0}^{t}{^{p}(k^{W}_{g_{s}})d|X^{\delta}_{s}|}+R_{t},$$
where $R$ is a continuous process with finite variations such that $R_{0}=0$ and $dR$ is carried by ${\{t\geq0:X^{\delta}_{t}=0\}}$. Since $W$ is continuous, we have: $^{p}(k^{W}_{g_{s}})=k^{W}_{g_{s-}}=k^{W}_{s-}$. Thus, it follows from the continuity of $|X^{\delta}|$ that
$$k^{W}_{g_{t}}X^{\delta}_{t}=\int_{0}^{t}{k^{W}_{s}d|X^{\delta}_{s}|}+R_{t}.$$
However, we have from Tanaka's formula that
$$d|X^{\delta}_{s}|=sign(X^{\delta}_{s})\sqrt{1-\delta^{2}}dB_{s}+\delta sign(X^{\delta}_{s})sign(W_{s})dW_{s}+\delta sign(X^{\delta}_{s})dL_{s}^{0}(W)+dL_{s}^{0}(X^{\delta}).$$
Hence, we obtain:
$$k^{W}_{g_{t}}|X^{\delta}_{t}|=\sqrt{1-\delta^{2}}\int_{0}^{t}{sign(X^{\delta}_{s})k^{W}_{s}dB_{s}}+\delta\int_{0}^{t}{k^{W}_{s}sign(X^{\delta}_{s})[sign(W_{s})dW_{s}+dL_{s}^{0}(W)]}+\int_{0}^{t}{k^{W}_{s}dL_{s}^{0}(X^{\delta})}+R_{t}.$$
Which becomes
$$k^{W}_{g_{t}}|X^{\delta}_{t}|=\sqrt{1-\delta^{2}}\int_{0}^{t}{sign(X^{\delta}_{s})k^{W}_{s}dB_{s}}+\delta\int_{0}^{t}{k^{W}_{s}sign(X^{\delta}_{s})sign(W_{s})dW_{s}}+\int_{0}^{t}{k^{W}_{s}dL_{s}^{0}(X^{\delta})}+R_{t}$$
since 
$$\int_{0}^{t}{k^{W}_{s}sign(X^{\delta}_{s})dL_{s}^{0}(W)}=\int_{0}^{t}{k^{W}_{s-}sign(X^{\delta}_{s})dL_{s}^{0}(W)}=\int_{0}^{t}{Z^{W}_{s}sign(X^{\delta}_{s})dL_{s}^{0}(W)}=0.$$ 
Indeed, $L^{0}(W)$ is continuous and $dL^{0}(W)$ is carried by ${\{t\geq0:W_{t}=0\}}={\{t\geq0:Z^{W}_{t}=0\}}$. Hence, through Proposition 2.2 of \cite{siam}, we get
 $$Y^{\delta}_{1,t}=\sqrt{1-\delta^{2}}\int_{0}^{t}{sign(X^{\delta}_{s})Z_{1,s}k^{W}_{s}dB_{s}}+\delta\int_{0}^{t}{sign(X^{\delta}_{s})Z_{1,s}k^{W}_{s}sign(W_{s})dW_{s}}+\int_{0}^{t}{Z_{1,s}[k^{W}_{s}dL_{s}^{0}(X^{\delta})+dR_{s}]}$$
$$\hspace{-12cm}+(2\alpha-1)L_{t}^{0}(Y^{\delta}_{1,\cdot}).$$
But, $dR$ and $dL^{0}(X^{\delta})$ are carried by ${\{t\geq0:X^{\delta}_{t}=0\}}$. In addition, 
$$\{t\geq0:X^{\delta}_{t}=0\}\subset\{t\geq0:k^{W}_{g_{t}}X^{\delta}_{t}=0\}=\{t\geq0:Z_{1,t}=0\}.$$
Therefore, we get:
$$Y^{\delta}_{1,t}=\sqrt{1-\delta^{2}}\int_{0}^{t}{sign(X^{\delta}_{s})Z_{1,s}k^{W}_{s}dB_{s}}+\delta\int_{0}^{t}{sign(X^{\delta}_{s})Z_{1,s}k^{W}_{s}sign(W_{s})dW_{s}}+(2\alpha-1)L_{t}^{0}(Y^{\delta}_{1,\cdot}).$$
Now, remark that the process $M$ defined by $\forall t\geq0$, 
$$M_{t}=\sqrt{1-\delta^{2}}\int_{0}^{t}{sign(X^{\delta}_{s})Z_{1,s}k^{W}_{s}dB_{s}}+\delta\int_{0}^{t}{sign(X^{\delta}_{s})Z_{1,s}k^{W}_{s}sign(W_{s})dW_{s}}$$
 is a continuous local martingale. In addition, we can see that $Z_{1,t}k^{W}_{g_{t}}|X^{\delta}_{t}|=k_{1,g_{t}}k^{W}_{g_{t}}|X^{\delta}_{t}|$ where $k_{1,\cdot}$ is the progressive process defined in \eqref{kalpha} with respect to the process $(k^{W}_{g_{t}}|X^{\delta}_{t}|:t\geq0)$. Hence by applying the balayage formula in progressive case on $k_{1,g_{t}}[k^{W}_{g_{t}}|X^{\delta}_{t}|]$, we obtain by identification that
$$M_{t}=\sqrt{1-\delta^{2}}\int_{0}^{t}{sign(X^{\delta}_{s})k_{1,s}k^{W}_{s}dB_{s}}+\delta\int_{0}^{t}{sign(X^{\delta}_{s})k_{1,s}k^{W}_{s}sign(W_{s})dW_{s}}$$
On another hand, we have:
$$\langle M,M\rangle_{t}=(1-\delta^{2})\int_{0}^{t}{\left(sign(X^{\delta}_{s})k_{1,s}k^{W}_{s}\right)^{2}ds}+\delta^{2}\int_{0}^{t}{\left((sign(X^{\delta}_{s})k_{1,s}k^{W}_{s}sign(W_{s})\right)^{2}ds}.$$
Which implies: $\langle M,M\rangle_{t}=t$ because $k_{1,s}\in\{-1,1\}$, $k^{W}_{s}\in\{-1,1\}$, $sign(X^{\delta}_{s})\in\{-1,1\}$ and $sign(W_{s})\in\{-1,1\}$. Consequently, $M$ is a Brownian motion. This completes the proof.
\end{proof}

Now, we shall provide a solution for \eqref{2} when $\alpha$ is a piecewise constant function associated with a partition $(0=t_{0}<t_{1}<\cdots<t_{n-1}<t_{m})$, i.e., $\alpha$ is of the form
$$\alpha(t)=\sum_{i=0}^{m}{\alpha_{i}1_{[t_{i},t_{i+1})}(t)},$$
where $\alpha_{i}\in[0,1]$ for all $i=0,1,\cdots,m$.

\begin{prop}
The process $Y^{\delta}_{2,\cdot}$ defined by $\forall t\geq0$, $Y^{\delta}_{2,t}=Z_{2,t}k^{W}_{g_{t}}|X^{\delta}_{t}|$ is a weak solution of \eqref{2} with the parameter $\alpha$ and starting from 0. 
\end{prop}
\begin{proof}
We have already shown in the above  last proof that
$$k^{W}_{g_{t}}|X^{\delta}_{t}|=\sqrt{1-\delta^{2}}\int_{0}^{t}{sign(X^{\delta}_{s})k^{W}_{s}dB_{s}}+\delta\int_{0}^{t}{k^{W}_{s}sign(X^{\delta}_{s})sign(W_{s})dW_{s}}+\int_{0}^{t}{k^{W}_{s}dL_{s}^{0}(X^{\delta})}+R_{t}.$$
Hence, we get from Proposition 2.3 of \cite{siam} that
 $$Y^{\delta}_{2,t}=\sqrt{1-\delta^{2}}\int_{0}^{t}{sign(X^{\delta}_{s})Z_{2,s}k^{W}_{s}dB_{s}}+\delta\int_{0}^{t}{sign(X^{\delta}_{s})Z_{2,s}k^{W}_{s}sign(W_{s})dW_{s}}+\int_{0}^{t}{Z_{2,s}[k^{W}_{s}dL_{s}^{0}(X^{\delta})+dR_{s}]}$$
$$\hspace{-10cm}+\int_{0}^{t}{(2\alpha(s)-1)dL_{s}^{0}(Y^{\delta}_{2,\cdot})}.$$
But, $dL_{\cdot}^{0}(X^{\delta})$ and $dR$ are carried by ${\{t\geq0:X^{\delta}_{t}=0\}}$. However, 
$$\{t\geq0:X^{\delta}_{t}=0\}\subset\{t\geq0:k^{W}_{g_{t}}X^{\delta}_{t}=0\}=\{t\geq0:Z_{2,t}=0\}.$$
Therefore,
$$Y^{\delta}_{2,t}=\sqrt{1-\delta^{2}}\int_{0}^{t}{sign(X^{\delta}_{s})Z_{2,s}k^{W}_{s}dB_{s}}+\delta\int_{0}^{t}{sign(X^{\delta}_{s})Z_{2,s}k^{W}_{s}sign(W_{s})dW_{s}}+\int_{0}^{t}{(2\alpha(s)-1)dL_{s}^{0}(Y^{\delta}_{2,\cdot})}.$$
Now, remark that the process $M^{'}$ defined by $\forall t\geq0$, 
$$M^{'}_{t}=\sqrt{1-\delta^{2}}\int_{0}^{t}{sign(X^{\delta}_{s})Z_{2,s}k^{W}_{s}dB_{s}}+\delta\int_{0}^{t}{sign(X^{\delta}_{s})Z_{2,s}k^{W}_{s}sign(W_{s})dW_{s}}$$
is a continuous local martingale. In addition, we can see that $Z_{2,t}k^{W}_{g_{t}}|X^{\delta}_{t}|=k_{2,g_{t}}k^{W}_{g_{t}}|X^{\delta}_{t}|$ where $k_{2,\cdot}$ is the progressive process defined in \eqref{Kalpha} with respect to the process $(k^{W}_{g_{t}}|X^{\delta}_{t}|:t\geq0)$. Hence by applying the balayage formula in the progressive case on $k_{2,g_{t}}[k^{W}_{g_{t}}|X^{\delta}_{t}|]$, we obtain by identification that
$$M^{'}_{t}=\sqrt{1-\delta^{2}}\int_{0}^{t}{sign(X^{\delta}_{s})k_{2,s}k^{W}_{s}dB_{s}}+\delta\int_{0}^{t}{sign(X^{\delta}_{s})k_{2,s}k^{W}_{s}sign(W_{s})dW_{s}}$$
On another hand, we have:
$$\langle M^{'},M^{'}\rangle_{t}=(1-\delta^{2})\int_{0}^{t}{\left(sign(X^{\delta}_{s})k_{2,s}k^{W}_{s}\right)^{2}ds}+\delta^{2}\int_{0}^{t}{\left((sign(X^{\delta}_{s})k_{2,s}k^{W}_{s}sign(W_{s})\right)^{2}ds}.$$
Which implies: $\langle M^{'},M^{'}\rangle_{t}=t$ because $k_{2,s}\in\{-1,1\}$, $k^{W}_{s}\in\{-1,1\}$, $sign(X^{\delta}_{s})\in\{-1,1\}$ and $sign(W_{s})\in\{-1,1\}$. Consequently, $M^{'}$ is a Brownian motion. This completes the proof.
\end{proof}

%%%%%%%%%%%%%%%%%%%%%%%%%%%%%%%%%%%%%%%%%%%%%%%%%%%%%%%%%%%%%%%%%%%%%%%%%%%%%%%%%%%%%%%%%
\subsection{Construction of solutions from a continuous process of the class \texorpdfstring{$\Sigma^{r}(H)$}{Sigma(r)}}

Now, we shall derive solutions by using continuous processes of the class $\Sigma^{r}(H)$. Thus, we shall use next notations: for any continuous process $X$ of the last mentioned class, we let $g_{t}=\sup\{s\leq t:X_{t}=0\}$ and $\tau_{t}=\inf\{s\geq0:\langle X,X\rangle_{s}>t\}$. Let $k^{D}$ and $k_{1,\cdot}$ be progressive processes defined in \eqref{kalpha} with respect to $D$ and $(k^{D}_{g_{\tau_{t}}}X_{\tau_{t}}:t\geq0)$ respectively. $Z_{2}$ is the progressive process defined in \eqref{Zalpha} with respect to $(k^{D}_{g_{\tau_{t}}}X_{\tau_{t}}:t\geq0)$. $k_{2}$ will denotes the progressive process defined in \eqref{Kalpha} with respect to $(k^{D}_{g_{\tau_{t}}}X_{\tau_{t}}:t\geq0)$. 

\begin{prop}
The process $\mathcal{Y}_{1}$ defined by $\forall t\geq0$, $\mathcal{Y}_{1,t}=Z_{1,t}k^{D}_{g_{\tau_{t}}}X_{\tau_{t}}$ is a weak solution of \eqref{1} with the parameter $\alpha$ and starting from 0. 
\end{prop}
\begin{proof}
Let $X=m+v+A$ be a continuous process of the class $\Sigma^{r}(H)$. By applying the balayage formula in the progressive case, we get
$$k^{D}_{g_{t}}X_{t}=\int_{0}^{t}{^{p}(k^{D}_{g_{s}})dX_{s}}+R_{t},$$
where $R$ is a continuous process with finite variations such that $R_{0}=0$ and $dR$ is carried by ${\{t\geq0:X_{t}=0\}}$. Since $D$ is continuous, we have: $^{p}(k^{D}_{g_{s}})=k^{D}_{g_{s-}}=k^{D}_{s-}$. Thus, it follows from the continuity of $X$ that
$$k^{D}_{g_{t}}X_{t}=\int_{0}^{t}{k^{D}_{s}dX_{s}}+R_{t}=\int_{0}^{t}{k^{D}_{s}dm_{s}}+\int_{0}^{t}{k^{D}_{s}dv_{s}}+\int_{0}^{t}{k^{D}_{s}dA_{s}}+R_{t}.$$
However, we have 
$$\int_{0}^{t}{k^{D}_{s}dv_{s}}=\int_{0}^{t}{Z^{D}_{s}dv_{s}},$$
where $Z^{D}$ is the progressive process defined in \eqref{zalpha} with respect to $D$. Thus, we obtain that
$$\int_{0}^{t}{k^{D}_{s}dv_{s}}=0$$
because $dv$ is carried by $\{t\geq0:D_{t}=0\}=\{t\geq0:Z^{D}_{t}=0\}$.
Which becomes
\begin{equation}\label{ee}
k^{D}_{g_{t}}X_{t}=\int_{0}^{t}{k^{D}_{s}dm_{s}}+\int_{0}^{t}{k^{D}_{s}dA_{s}}+R_{t}.
\end{equation}
Now, let $Y_{t}=k^{D}_{g_{\tau_{t}}}X_{\tau_{t}}$. We obtain by applying Proposition 2.2 of \cite{siam}, the following:
$$\mathcal{Y}_{1,t}=Z_{1,t}Y_{t}=\int_{0}^{t}{Z_{1,s}dY_{s}}+(2\alpha-1)L_{t}^{0}(\mathcal{Y}_{1,\cdot}).$$
But,
$$\int_{0}^{t}{Z_{1,s}dY_{s}}=\int_{0}^{t}{Z_{1,s}k^{D}_{\tau_{s}}dm_{\tau_{s}}}+\int_{0}^{t}{Z_{1,s}k^{D}_{\tau_{s}}dA_{\tau_{s}}}+\int_{0}^{t}{Z_{1,s}dR_{\tau_{s}}}.$$
Hence, we obtain:
$$\int_{0}^{t}{Z_{1,s}dY_{s}}=\int_{0}^{t}{Z_{1,s}k^{D}_{\tau_{s}}dm_{\tau_{s}}}$$
because, $dA_{\tau_{\cdot}}$ and $dR_{\tau_{\cdot}}$ are carried by $\{t\geq0: Z_{1,t}=0\}$. Which implies that
$$\mathcal{Y}_{1,t}=\int_{0}^{t}{Z_{1,s}k^{D}_{\tau_{s}}dm_{\tau_{s}}}+(2\alpha-1)L_{t}^{0}(\mathcal{Y}_{1,\cdot}).$$
However, the process $W$ defined by $\forall t\geq0$, 
$$W_{t}=\int_{0}^{t}{Z_{1,s}k^{D}_{\tau_{s}}dm_{\tau_{s}}}=\int_{0}^{\tau_{t}}{Z_{1,\langle X,X\rangle_{s}}k^{D}_{s}dm_{s}}$$
is a continuous local martingale. Furthermore, by applying the balayage formula on $k_{1,g^{'}_{t}}Y_{t}$, we obtain by identification that
$$W_{t}=\int_{0}^{\tau_{t}}{k_{1,\langle X,X\rangle_{s}}k^{D}_{s}dm_{s}}$$
since  $Z_{1,t}Y_{t}=k_{1,g^{'}_{t}}Y_{t}$ with $g^{'}_{t}=\sup\{s\leq t:Y_{t}=0\}$. On another hand, we have:
$$\langle W,W\rangle_{t}=\int_{0}^{\tau_{t}}{\left(k_{1,\langle X,X\rangle_{s}}k^{D}_{s}\right)^{2}d\langle m,m\rangle_{s}}=\langle m,m\rangle_{\tau_{t}}=t.$$
Consequently, $W$ is a Brownian motion. This completes the proof.
\end{proof}

\begin{prop}
The process $\mathcal{Y}_{2}$ defined by $\forall t\geq0$, $\mathcal{Y}^{2}_{t}=Z_{2,t}k^{D}_{g_{\tau_{t}}}X_{\tau_{t}}$ is a weak solution of \eqref{2} with the parameter $\alpha$ and starting from 0. 
\end{prop}
\begin{proof}
Recall that we have obtained in \eqref{ee}, the following:
$$k^{D}_{g_{t}}X_{t}=\int_{0}^{t}{k^{D}_{s}dm_{s}}+\int_{0}^{t}{k^{D}_{s}dA_{s}}+R_{t}.$$
If we let $Y_{t}=k^{D}_{g_{\tau_{t}}}X_{\tau_{t}}$, we obtain from Proposition 2.3 of \cite{siam}, the following:
$$\mathcal{Y}_{2,t}=Z_{2,t}Y_{t}=\int_{0}^{t}{Z_{2,s}dY_{s}}+\int_{0}^{t}{(2\alpha(s)-1)dL_{s}^{0}(\mathcal{Y}_{2,\cdot})}.$$
But,
$$\int_{0}^{t}{Z_{2,s}dY_{s}}=\int_{0}^{t}{Z_{2,s}k^{D}_{\tau_{s}}dm_{\tau_{s}}}+\int_{0}^{t}{Z_{2,s}k^{D}_{\tau_{s}}dA_{\tau_{s}}}+\int_{0}^{t}{Z_{2,s}dR_{\tau_{s}}}.$$
Hence, we obtain:
$$\int_{0}^{t}{Z_{2,s}dY_{s}}=\int_{0}^{t}{Z_{2,s}k^{D}_{\tau_{s}}dm_{\tau_{s}}}$$
because, $dA_{\tau_{\cdot}}$ and $dR_{\tau_{\cdot}}$ are carried by $\{t\geq0: Z_{2,t}=0\}$. Which implies that
$$\mathcal{Y}_{2,t}=\int_{0}^{t}{Z_{2,s}k^{D}_{\tau_{s}}dm_{\tau_{s}}}+\int_{0}^{t}{(2\alpha(s)-1)dL_{s}^{0}(\mathcal{Y}_{2,\cdot})}.$$
However, the process $W$ defined by $\forall t\geq0$, 
$$W_{t}=\int_{0}^{t}{Z_{2,s}k^{D}_{\tau_{s}}dm_{\tau_{s}}}=\int_{0}^{\tau_{t}}{Z_{2,\langle X,X\rangle_{s}}k^{D}_{s}dm_{s}}$$
is a continuous local martingale. Furthermore, by applying the balayage formula on $k_{2,g^{'}_{t}}Y_{t}$, we obtain by identification that
$$W_{t}=\int_{0}^{\tau_{t}}{k_{2,\langle X,X\rangle_{s}}k^{D}_{s}dm_{s}}$$
since  $Z_{2,t}Y_{t}=k_{2,g^{'}_{t}}Y_{t}$ with $g^{'}_{t}=\sup\{s\leq t:Y_{t}=0\}$. On another hand, we have:
$$\langle W,W\rangle_{t}=\int_{0}^{\tau_{t}}{\left(k_{2,\langle X,X\rangle_{s}}k^{D}_{s}\right)^{2}d\langle m,m\rangle_{s}}=\langle m,m\rangle_{\tau_{t}}=t.$$
Consequently, $W$ is a Brownian motion. This completes the proof.
\end{proof}
%%%%%%%%%%%%%%%%%%%%%%%%%%%%%%%%%%%%%%%%%%%%%%%%%%%%%%%%%%%%%%%%%%%%%%%%%%%%%%%%%%%%%%%%%%%%%%%%%%%%%%%%%%%%%%%%%%%%%%%%%%%%%%%%%%%%%%%%%%%%%%%%%%%%%%%%%%%%%%%%%%%%%%

{\color{myaqua}}
\end{document}